\documentclass[12pt]{amsart}       
\usepackage{txfonts}
\usepackage{amssymb}
\usepackage{eucal}
\usepackage{graphicx}
\usepackage{amsmath}
\usepackage{amscd}
\usepackage[all]{xy}           

\usepackage{amsfonts,latexsym}
\usepackage{xspace}
\usepackage{epsfig}
\usepackage{float}
\usepackage{color}
\usepackage{fancybox}
\usepackage{colordvi}
\usepackage{multicol}
\usepackage{colordvi}
\usepackage{tikz}

\usepackage{wasysym}
\usepackage[active]{srcltx} 
\ifpdf
  \usepackage[colorlinks,final,backref=page,hyperindex]{hyperref}
\else
  \usepackage[colorlinks,final,backref=page,hyperindex,hypertex]{hyperref}
\fi

\newcommand{\nc}{\newcommand}
\newcommand{\delete}[1]{}

\nc{\mlabel}[1]{\label{#1}}  
\nc{\mcite}[1]{\cite{#1}}  
\nc{\mref}[1]{\ref{#1}}  
\nc{\meqref}[1]{\eqref{#1}}  
\nc{\mbibitem}[1]{\bibitem{#1}} 

\delete{
\nc{\mlabel}[1]{\label{#1}  
{\hfill \hspace{1cm}{\small\tt{{\ }\hfill(#1)}}}}
\nc{\mcite}[1]{\cite{#1}{\small{\tt{{\ }(#1)}}}}  
\nc{\mref}[1]{\ref{#1}{{\tt{{\ }(#1)}}}}  
\nc{\meqref}[1]{\eqref{#1}{{\tt{{\ }(#1)}}}}  
\nc{\mbibitem}[1]{\bibitem[\bf #1]{#1}} 
}

\newtheorem{theorem}{Theorem}[section]
\newtheorem{lemma}[theorem]{Lemma}

\theoremstyle{definition}
\newtheorem{defn}[theorem]{Definition}
\newtheorem{prop-def}{Proposition-Definition}[section]

\newtheorem{remark}[theorem]{Remark}

\newtheorem{tempex}[theorem]{Example}
\newtheorem{tempexs}[theorem]{Examples}
\newtheorem{temprmk}[theorem]{Remark}
\newtheorem{tempexer}{Exercise}[section]
\newenvironment{exam}{\begin{tempex}\rm}{\end{tempex}}

\nc{\vsa}{\vspace{-.1cm}} \nc{\vsb}{\vspace{-.2cm}}
\nc{\vsc}{\vspace{-.3cm}} \nc{\vsd}{\vspace{-.4cm}}
\nc{\vse}{\vspace{-.5cm}}

\nc{\Irr}{\mathrm{Irr}}
\nc{\ncrbw}{\calr}  
\nc{\NS}{U_{NS}}
\nc{\FN}{F_{\mathrm Nij}}
\nc{\dfgen}{V} \nc{\dfrel}{R}
\nc{\dfgenb}{\vec{v}} \nc{\dfrelb}{\vec{r}}
\nc{\dfgene}{v} \nc{\dfrele}{r}
\nc{\dfop}{\odot}
\nc{\dfoa}{\dfop^{(1)}} \nc{\dfob}{\dfop^{(2)}}
\nc{\dfoc}{\dfop^{(3)}} \nc{\dfod}{\dfop^{(4)}}
\nc{\mapm}[1]{\frakM{(#1)}}\nc{\maps}[1]{\frakS{(#1)}}
\nc{\cmapm}[1]{\frakC(#1)}
\nc{\red}{\mathrm{Red}}
\nc{\cm}{C}
\nc{\supp}{\mathrm{Supp}}
\nc{\lex}{\mathrm{lex}}

\nc{\disp}[1]{\displaystyle{#1}}
\nc{\bin}[2]{ (_{\stackrel{\scs{#1}}{\scs{#2}}})}  
\nc{\binc}[2]{ \left (\!\! \begin{array}{c} \scs{#1}\\
    \scs{#2} \end{array}\!\! \right )}  
\nc{\bincc}[2]{  \left ( {\scs{#1} \atop
    \vspace{-.5cm}\scs{#2}} \right )}  
\nc{\sarray}[2]{\begin{array}{c}#1 \vspace{.1cm}\\ \hline
    \vspace{-.35cm} \\ #2 \end{array}}
\nc{\bs}{\bar{S}} \nc{\ep}{\epsilon}
\nc{\dbigcup}{\stackrel{\bullet}{\bigcup}}
\nc{\la}{\longrightarrow} \nc{\cprod}{\ast} \nc{\rar}{\rightarrow}
\nc{\dar}{\downarrow} \nc{\labeq}[1]{\stackrel{#1}{=}}
\nc{\dap}[1]{\downarrow \rlap{$\scriptstyle{#1}$}}
\nc{\uap}[1]{\uparrow \rlap{$\scriptstyle{#1}$}}
\nc{\defeq}{\stackrel{\rm def}{=}} \nc{\dis}[1]{\displaystyle{#1}}
\nc{\dotcup}{\ \displaystyle{\bigcup^\bullet}\ }
\nc{\sdotcup}{\tiny{ \displaystyle{\bigcup^\bullet}\ }}
\nc{\fe}{\'{e}}
\nc{\hcm}{\ \hat{,}\ } \nc{\hcirc}{\hat{\circ}}
\nc{\hts}{\hat{\shpr}} \nc{\lts}{\stackrel{\leftarrow}{\shpr}}
\nc{\denshpr}{\den{\shpr}}
\nc{\rts}{\stackrel{\rightarrow}{\shpr}} \nc{\lleft}{[}
\nc{\lright}{]} \nc{\uni}[1]{\tilde{#1}} \nc{\free}[1]{\bar{#1}}
\nc{\freea}[1]{\tilde{#1}} \nc{\freev}[1]{\hat{#1}}
\nc{\dt}[1]{\hat{#1}}
\nc{\wor}[1]{\check{#1}}
\nc{\intg}[1]{F_C(#1)}
\nc{\den}[1]{\check{#1}} \nc{\lrpa}{\wr} \nc{\mprod}{\pm}
\nc{\dprod}{\ast_P} \nc{\curlyl}{\left \{ \begin{array}{c} {} \\
{} \end{array}
    \right .  \!\!\!\!\!\!\!}
\nc{\curlyr}{ \!\!\!\!\!\!\!
    \left . \begin{array}{c} {} \\ {} \end{array}
    \right \} }
\nc{\longmid}{\left | \begin{array}{c} {} \\ {} \end{array}
    \right . \!\!\!\!\!\!\!}
\nc{\lin}{\call} \nc{\ot}{\otimes}
\nc{\ora}[1]{\stackrel{#1}{\rar}}
\nc{\ola}[1]{\stackrel{#1}{\la}}
\nc{\scs}[1]{\scriptstyle{#1}} \nc{\mrm}[1]{{\rm #1}}
\nc{\margin}[1]{\marginpar{\rm #1}}   
\nc{\dirlim}{\displaystyle{\lim_{\longrightarrow}}\,}
\nc{\invlim}{\displaystyle{\lim_{\longleftarrow}}\,}
\nc{\mvp}{\vspace{0.5cm}}
\nc{\mult}{m}       
\nc{\svp}{\vspace{2cm}} \nc{\vp}{\vspace{8cm}}
\nc{\proofbegin}{\noindent{\bf Proof: }}
\nc{\proofend}{$\blacksquare$ \vspace{0.5cm}}
\nc{\sha}{{\mbox{\cyr X}}}  
\nc{\ncsha}{{\mbox{\cyr X}^{\mathrm NC}}}
\newfont{\scyr}{wncyr10 scaled 550}
\nc{\ssha}{\mbox{\bf \scyr X}}
\nc{\ncshao}{{\mbox{\cyr X}^{\mathrm NC,\,0}}}
\nc{\shpr}{\diamond}    
\nc{\shprc}{\shpr_c}
\nc{\shpro}{\diamond^0}    
\nc{\shpru}{\check{\diamond}} \nc{\spr}{\cdot}
\nc{\catpr}{\diamond_l} \nc{\rcatpr}{\diamond_r}
\nc{\lapr}{\diamond_a} \nc{\lepr}{\diamond_e} \nc{\sprod}{\bullet}
\nc{\un}{u}                 
\nc{\vep}{\varepsilon} \nc{\labs}{\mid\!} \nc{\rabs}{\!\mid}
\nc{\hsha}{\widehat{\sha}} \nc{\psha}{\sha^{+}} \nc{\tsha}{\tilde{\sha}}
\nc{\lsha}{\stackrel{\leftarrow}{\sha}}
\nc{\rsha}{\stackrel{\rightarrow}{\sha}} \nc{\lc}{\lfloor}
\nc{\rc}{\rfloor} \nc{\sqmon}[1]{\langle #1\rangle}
\nc{\altx}{\Lambda} \nc{\vecT}{\vec{T}} \nc{\piword}{{\mathfrak P}}
\nc{\lbar}[1]{\overline{#1}}
\nc{\dep}{\mathrm{dep}}

\nc{\mmbox}[1]{\mbox{\ #1\ }}
\nc{\ayb}{\mrm{AYB}} \nc{\mayb}{\mrm{mAYB}} \nc{\cyb}{\mrm{cyb}}
\nc{\ann}{\mrm{ann}} \nc{\Aut}{\mrm{Aut}} \nc{\cabqr}{\mrm{CABQR
}} \nc{\can}{\mrm{can}} \nc{\colim}{\mrm{colim}}
\nc{\Cont}{\mrm{Cont}} \nc{\rchar}{\mrm{char}}
\nc{\cok}{\mrm{coker}} \nc{\dtf}{{R-{\rm tf}}} \nc{\dtor}{{R-{\rm
tor}}}

\nc{\Div}{{\mrm Div}} \nc{\End}{\mrm{End}} \nc{\Ext}{\mrm{Ext}}
\nc{\FG}{\mrm{FG}} \nc{\Fil}{\mrm{Fil}} \nc{\Frob}{\mrm{Frob}}
\nc{\Gal}{\mrm{Gal}} \nc{\GL}{\mrm{GL}} \nc{\Hom}{\mrm{Hom}}
\nc{\hsr}{\mrm{H}} \nc{\hpol}{\mrm{HP}} \nc{\id}{\mrm{id}} \nc{\Id}{\mathrm{Id}}  \nc{\ID}{\mathrm{ID}}
\nc{\im}{\mrm{im}} \nc{\incl}{\mrm{incl}} \nc{\Loday}{\mrm{ABQR}\
} \nc{\length}{\mrm{length}} \nc{\LR}{\mrm{LR}} \nc{\mchar}{\rm
char} \nc{\pmchar}{\partial\mchar} \nc{\map}{\mrm{Map}}
\nc{\MS}{\mrm{MS}} \nc{\OS}{\mrm{OS}} \nc{\NC}{\mrm{NC}}
\nc{\rba}{\rm{Rota-Baxter algebra}\xspace}
\nc{\rbas}{\rm{Rota-Baxter algebras}\xspace}
\nc{\rbw}{\ncrbw}
\nc{\rbws}{\rm{RBWs}\xspace}
\nc{\rbadj}{\rm{RB}\xspace}
\nc{\mpart}{\mrm{part}} \nc{\ql}{{\QQ_\ell}} \nc{\qp}{{\QQ_p}}
\nc{\rank}{\mrm{rank}} \nc{\rcot}{\mrm{cot}} \nc{\rdef}{\mrm{def}}
\nc{\rdiv}{{\rm div}} \nc{\rtf}{{\rm tf}} \nc{\rtor}{{\rm tor}}
\nc{\res}{\mrm{res}} \nc{\SL}{\mrm{SL}} \nc{\Spec}{\mrm{Spec}}
\nc{\tor}{\mrm{tor}} \nc{\Tr}{\mrm{Tr}}
\nc{\mtr}{\mrm{tr}}

\nc{\ab}{\mathbf{Ab}} \nc{\Alg}{\mathbf{Alg}}
\nc{\Bax}{\mathbf{CRB}} \nc{\Algo}{\mathbf{Alg}^0}
\nc{\cRB}{\mathbf{CRB}} \nc{\cRBo}{\mathbf{CRB}^0}
\nc{\RBo}{\mathbf{RB}^0} \nc{\BRB}{\mathbf{RB}}
\nc{\Dend}{\mathbf{DD}} \nc{\bfk}{{\bf k}} \nc{\bfone}{{\bf 1}}
\nc{\base}[1]{{a_{#1}}} \nc{\Cat}{\mathbf{Cat}}
 \nc{\DN}{\mathbf{DN}}
\nc{\NA}{\mathbf{NA}}
\nc{\SDN}{\mathbf{SDN}}
\nc{\Diff}{\mathbf{Diff}} \nc{\gap}{\marginpar{\bf
Incomplete}\noindent{\bf Incomplete!!}
    \svp}
\nc{\FMod}{\mathbf{FMod}} \nc{\Int}{\mathbf{Int}}
\nc{\Mon}{\mathbf{Mon}}
\nc{\RB}{\mathbf{RB}} \nc{\remarks}{\noindent{\bf Remarks: }}
\nc{\Rep}{\mathbf{Rep}} \nc{\Rings}{\mathbf{Rings}}
\nc{\Sets}{\mathbf{Sets}} \nc{\bfx}{\mathbf{x}}
\nc{\BA}{{\Bbb A}} \nc{\CC}{{\Bbb C}} \nc{\DD}{{\Bbb D}}
\nc{\EE}{{\Bbb E}} \nc{\FF}{{\Bbb F}} \nc{\GG}{{\Bbb G}}
\nc{\HH}{{\Bbb H}} \nc{\LL}{{\Bbb L}} \nc{\NN}{{\Bbb N}}
\nc{\QQ}{{\Bbb Q}} \nc{\RR}{{\Bbb R}} \nc{\TT}{{\Bbb T}}
\nc{\VV}{{\Bbb V}} \nc{\ZZ}{{\Bbb Z}}

\nc{\cala}{{\mathcal A}} \nc{\calb}{{\mathcal B}}
\nc{\calc}{{\mathcal C}}
\nc{\cald}{{\mathcal D}} \nc{\cale}{{\mathcal E}}
\nc{\calf}{{\mathcal F}} \nc{\calg}{{\mathcal G}}
\nc{\calh}{{\mathcal H}} \nc{\cali}{{\mathcal I}}
\nc{\calj}{{\mathcal J}} \nc{\call}{{\mathcal L}}
\nc{\calm}{{\mathcal M}} \nc{\caln}{{\mathcal N}}
\nc{\calo}{{\mathcal O}} \nc{\calp}{{\mathcal P}}
\nc{\calr}{{\mathcal R}} \nc{\cals}{{\mathcal S}} \nc{\calt}{{\mathcal T}}
\nc{\calw}{{\mathcal W}} \nc{\calx}{{\mathcal X}} \nc{\caly}{{\mathcal Y}} \nc{\calz}{{\mathcal Z}}
\nc{\CA}{\mathcal{A}}

\nc{\frakA}{{\mathfrak A}}
\nc{\fraka}{{\mathfrak a}}
\nc{\frakB}{{\mathfrak B}}
\nc{\frakb}{{\mathfrak b}}
\nc{\frakC}{{\mathfrak C}}
\nc{\frakd}{{\mathfrak d}}
\nc{\frakF}{{\mathfrak F}}
\nc{\frakg}{{\mathfrak g}}
\nc{\frakm}{{\mathfrak m}}
\nc{\frakM}{{\mathfrak M}}
\nc{\frakMo}{{\mathfrak M}^0}
\nc{\frakP}{{\mathfrak P}}
\nc{\frakp}{{\mathfrak p}}
\nc{\frakS}{{\mathfrak S}}
\nc{\frakSo}{{\mathfrak S}^0}
\nc{\fraks}{{\mathfrak s}}
\nc{\os}{\overline{\fraks}}
\nc{\frakT}{{\mathfrak T}}
\nc{\frakTo}{{\mathfrak T}^0}
\nc{\oT}{\overline{T}}
\nc{\frakX}{{\mathfrak X}}
\nc{\frakXo}{{\mathfrak X}^0}
\nc{\frakx}{{\mathbf x}}
\nc{\frakTx}{\frakT}      
\nc{\frakTa}{\frakT^a}        
\nc{\frakTxo}{\frakTx^0}   
\nc{\caltao}{\calt^{a,0}}   
\nc{\ox}{\overline{\frakx}} \nc{\fraky}{{\mathfrak y}}
\nc{\frakz}{{\mathfrak z}} \nc{\oX}{\overline{X}} \font\cyr=wncyr10

\nc{\tred}[1]{\textcolor{red}{#1}} \nc{\tgreen}[1]{\textcolor{green}{#1}}
\nc{\tblue}[1]{\textcolor{blue}{#1}} \nc{\tpurple}[1]{\textcolor{purple}{#1}}

\nc{\hu}[1]{\tpurple{\underline{Huhu:}#1 }}
\nc{\liadd}[1]{\tpurple{#1}}
\nc{\xing}[1]{\tblue{\underline{Xing:}#1 }}

\nc{\markus}[1]{\tred{\underline{Markus:} #1}}

\nc{\dnx}{\Delta_n X} \nc{\dx}{\Delta X} \nc{\dgp}{{\rm deg_{P}}}
\nc{\dgt}{{\rm deg_{T}}} \nc{\dg}{{\rm deg}} \nc{\ida}{ID($A$)} \nc{\tu}{\tilde{u}} \nc{\tv}{\tilde{v}}
\nc{\nr}{\calr_n} \nc{\nz}{\calz_n} \nc{\fun}{\cala_{n,d}}
 \nc{\fbase}{\calb} \nc{\LF}{\mathrm{RF}} \nc{\FFA}{\mathrm{LF}} \nc{\irr}{\mathrm{Irr}}
 \nc{\result}{\bfk\mathrm{Irr}(S_n)}  \nc{\I}{I_{\mathrm{ID},n}^0}
 \nc{\nrs}{\calr_n^\star} \nc{\ii}{\mathrm{I}} \nc{\iii}{\mathrm{II}}

\nc{\ws}[1]{{#1}} \nc{\deleted}[1]{\delete{#1}} \nc{\plas}{placements\xspace}
\nc{\mapmonoid}{\frakM} \nc\kdot{\bfk} \nc{\mstar}{\frakM^\star(X)}
\nc{\stars}[2]{#1|_{#2}} \nc{\medmid}{{\,~{\tiny \longmid}~\,}} \nc{\nbfk}{\bfk^{\times}}
\nc{\re}[1]{R(#1)}  \nc{\paren}[1]{$($#1$)$}
\nc\gsbs{  Gr\"{o}bner-Shirshov bases }\nc{\suba}[1]{|_{#1}}\nc\gsb{  Gr\"{o}bner-Shirshov basis }
\nc\bre{{\rm bre}_P}\nc\blw[1]{\lc#1\rc}\nc\degp{{\rm deg}_P}  \nc\ordr{\leq_{\rm rb}}
\nc{\dbl}{\leq_{\rm PL_l}} \nc{\dbr}{\leq_{\rm PL_r}}\nc{\dbln}{>_{\rm PL_l}} \nc{\dbrn}{>_{\rm PL_r}}
\nc\ordrnq{>_{\rm rb}}\nc\ordb{\leq_{\rm db}}
\nc\wti{{\rm wt}_n} \nc{\wt}{\text{wt}_{\rm r}} \nc\wtb{{\rm wt}_{\rm l}}
\nc\wtd{{\rm wt}}\nc{\clex}{\mathrm{clex}}\nc{\dlex}{\leq_{\rm dlex}}
\nc\pat{{\rm Path}} \nc{\patl}{{\rm Path_{\rm l}}} \nc{\patr}{{\rm Path_{\rm r}}}
\nc{\X}{X^\ast}\nc\degx{{\rm deg}_{\X}}

\begin{document}

\title[Free differential algebra and free Rota-Baxter algebra]{Free weighted (modified) differential algebras, free (modified) Rota-Baxter algebras and Gr\"{o}bner-Shirshov bases}

\author{Zhicheng Zhu}
\address{School of Mathematics and Statistics, Lanzhou University, Lanzhou, Gansu 730000, P.\,R. China}
\email{zhuzhch16@lzu.edu.cn}

\author{Huhu Zhang}
\address{School of Mathematics and Statistics,
	Lanzhou University, Lanzhou, Gansu 730000, P. R. China}
\email{zhanghh17@lzu.edu.cn}

\author{Xing Gao$^{*}$}\thanks{* Corresponding author.} \address{School of Mathematics and Statistics,
Lanzhou University, Lanzhou, 730000, P.R. China}
\email{gaoxing@lzu.edu.cn}

\date{\today}

\begin{abstract}
In this paper, we obtain respectively some new linear bases of free unitary (modified) weighted differential algebras and free nonunitary (modified) Rota-Baxter algebras, in terms of the method of Gr\"{o}bner-Shirshov bases.
\end{abstract}

\subjclass[2010]{
16W99, 
16S10, 
13P10, 
12H05, 
08B20, 
16T30,  
}

\keywords{Gr\"{o}bner-Shirshov bases; free (modified) Rota-Baxter algebras; free weighted (modified) differential algebras}

\maketitle

\tableofcontents

\setcounter{section}{0}

\allowdisplaybreaks

\section{Introduction}
The free object is the most significant object in a category.
The free (modified) weighted differential algebra and the free (modified) Rota-Baxter algebra
have been constructed before. The aim of the present paper is to establish two new monomial orders
on bracketed words and then to obtain some new linear bases of free unitary weighted (modified) differential algebras and free nonunitary (modified) Rota-Baxter algebras on sets, using the method of Gr\"{o}bner-Shirshov bases.

Throughout this paper, let $\bfk$ denote a unitary commutative ring of characteristic zero, which will be the base ring of all modules, algebras,  tensor products, as well as linear maps.
Unless otherwise stated, all algebras are associative and noncommutative, and by an algebra we mean a unitary algebra.

\subsection{Gr\"{o}bner-Shirshov bases}
The theories of  Gr\"{o}bner-Shirshov bases and  Gr\"{o}bner bases were invented independently
by Shirshov for anti-commutative and nonassociative algebras in 1962~\cite{Sh}, by Hironaka~\mcite{Hi}(1964) and Buchberger~\mcite{Bu} (1965) for polynomial algebras. Later Bokut, Chen and Guo et al. studied the Gr\"{o}bner-Shirshov bases for associative algebras with multiple operators~\mcite{BCQ,GSZ} and constructed a lot of free objects as applications. See~\mcite{BC} and references therein for more details.

\subsection{(Modified) weighted differential algebras}
The differential algebra, introduced by Ritt~\mcite{R50}, is an algebraic approach to differential equations replacing analytic notions like differential quotient by Leibniz rule. Later due to the work of Kolchin and many other mathematicians, the differential algebra
has broad applications in mathematics and physics, such as algebraic group~\mcite{K73}, category~\mcite{PJ19}, Galois theory~\mcite{SP03}, operad~\mcite{GK94}, Poisson Hopf algebra~\mcite{GHL18} and representation theory~\mcite{S17}.

The notion of a differential algebra of weight $\lambda$ was first invented by Guo and Keigher~\mcite{GK},
defined to be an algebra equipped with a linear operator $D$ satisfying $D(1) = 0$ and the {\bf weighted Leibniz rule}
$$D(xy)=D(x)y+xD(y)+\lambda D(x)D(y),$$
which generalizes simultaneously the concept of the classical differential algebra and difference algebra~\mcite{GRR}.
Applying the same method as for free differential algebras (of weight 0), free differential algebras of weight $\lambda$ on sets
were constructed in both the commutative case and the noncommutative case~\mcite{GK}, taking the monomial $D(xy)$ as the role of leading monomial. In the present paper, we reconstruct free differential algebras of weight $\lambda$ on sets by choosing other monomials as leading monomial, via the method of Gr\"{o}bner-Shirshov bases.

A Rota-Baxter operator (see below) gives a modified Rota-Baxter operator (see below) through a linear transformation~\mcite{E}.
More precisely, from a Rota-Baxter operator $P$ of weight $\lambda$, we obtain a modified Rota-Baxter operator $Q$ of weight $-\lambda^2$
by a linear transformation: $Q :=-2P - \lambda \id$.
Motivated by this, the concept of weighted modified differential algebra
was introduced in~\mcite{PZGL}, by taking a linear transformation of a differential operator $d$.
Namely, starting from a differential operator $d$, the operator $D := d - \lambda \id$
is a modified differential operator of weight $\lambda$ satisfying
\begin{equation*}
D(xy)= D(x)y + x D(y)+ \lambda xy.
\end{equation*}
The modified differential operator of weight $\lambda$ is a kind of differential type operator, and the corresponding free object was constructed in~\mcite{GSZ}, taking the monomial $D(xy)$ as the leading monomial again.
Parrel to the differential case, we reconstruct free weighted modified differential algebras by taking other monomials
as leading monomial in this paper.

\subsection{(Modified) Rota-Baxter algebras}
The concept of a Rota-Baxter algebra was originated from Baxter's probability study to understand  Spitzer's identity in fluctuation theory~\mcite{Ba}. In that paper, Baxter was led to a
linear operator $P$ on an algebra of functions satisfying the {\bf Rota-Baxter identity}
$$P(x)P(y)=P(P(x)y+xP(y)+\lambda xy).$$
This was further developed by Rota~\mcite{Ro} and Cartier~\mcite{Ca}, concerning the relationship with combinatorics.
After a period of relatively little activity, Rota-Baxter algebra has enjoyed a resurgence in the new century. Now it has remarkable applications in mathematics and physics~\mcite{Ag,BBGN,BG,FG,GK1,PBG,GLS}.
In particular, it was found crucial applications
in the Connes-Kreimer's algebraic approach to renormalization of quantum field theory~\mcite{CK1,CK2}.

The construction of free Rota-Baxter algebras can be divided into the commutative case and the noncommutative case.
For the commutative case, Rota~\mcite{Ro} and Cartier~\mcite{Ca} gave explicit constructions of the free commutative nonunitary Rota-Baxter algebra of weight -1 on a set. Later these constructions were generalized by Guo and Keigher~\mcite{GK1}. There Guo and Keigher generalized the classical shuffle product to the mixable shuffle product and constructed the free commutative unitary Rota-Baxter algebra of any weight $\lambda$ on an algebra, containing the free Rota-Baxter algebra of Rota or Cartier as a sub-Rota-Baxter algebra.

For the noncommutative case, free noncommutative Rota-Baxter algebras of arbitrary weight $\lambda$ were constructed in both the category of unitary Rota-Baxter algebras and that of nonunit Rota-Baxter algebras~\mcite{FG,FG08,Gop}, utilizing some combinatorial objects such as  bracketed words and planar rooted trees, with the later being related rooted trees of Connes-Kreimer~\mcite{CK1} and planar binary trees of Loday-Ronco~\mcite{LR98}. In addition, Bokut et al.~\mcite{BCQ} established the Composition-Diamond lemma for nonunitary algebras with multiple linear operators and obtained a Gr\"{o}bner-Shirshov basis of the free nonunitary Rota-Baxter algebra. More generally, Guo et al. studied the Rota-Baxter type algebra and constructed the free object uniformly in that category~\cite{ZGGS}.

In all of the above constructions, the monomial $P(x)P(y)$ in the Rota-Baxter identity is treated as the role of leading monomial.
A natural question is that can we choose other monomials in the Rota-Baxter identity as the leading monomial to construct free Rota-Baxter algebras. This is carried out affirmatively in the present paper via the method of Gr\"{o}bner-Shirshov bases, focusing on the noncommutative nonunitary case.

In the 1980s, Semonov-Tian-Shansky~\cite{Sem} found that, under suitable conditions, $r$ is a solution
of the classical Yang-Baxter equation if and only if the corresponding operator $P$ is a Rota-Baxter
operator of weight zero
\begin{equation*}
[P(x),P(y)]=P[P(x),y]+P[x,P(y)]
\end{equation*}
on a Lie algebra.
As a modified form of the operator form of the classical Yang-Baxter equation,
he also initiated in that paper the {\bf modified classical Yang-Baxter equation}
\begin{equation*}
[P(x),P(y)]=P[P(x),y]+P[x,P(y)]- [x, y],
\end{equation*}
which has connections with generalized Lax pairs and affine geometry on Lie groups~\mcite{BGN1,Bo,KS}.
As the associative version, the {\bf modified associative Yang-Baxter equation}
\begin{equation*}
P(x)  P(y) = P (P(x)y) + P(xP(y))- xy.
\end{equation*}
has been applied to the study of extended $\calo$-operators, associative Yang-Baxter equations, infinitesimal bialgebras and dendriform algebras~\mcite{BGN2}. Quite recently, Guo et al. ~\mcite{ZGG19} constructed the free modified Rota-Baxter algebra via bracketed words, taking $P(x)P(y)$ as the role of leading monomial. In this paper, we also reconstructed the free nonunitary modified Rota-Baxter algebra, in terms of the method of Gr\"{o}bner-Shirshov bases by choosing other monomials as leading monomial.

\subsection{Structure of the paper} In Section 2, we review the construction of the free operated algebra on a set and some
basic notations of the theory of Gr\"{o}bner-Shirshov bases. In particular, the cornerstone Composition-Diamond lemma is recalled (Lemmas~\mref{lem:cdl} and~\mref{lem:cdlnu}), which will be used frequently in the paper.
Section 3 is devoted to some new linear bases of the free (modified) differential algebra.
We first define two new monomial orders on bracketed words (Theorem~\mref{thm:rbord}).
Based on these two monomial orders, we then acquire two new linear bases of the free differential algebra of weight zero (Theorem~\mref{thm:gsbd0} and Theorem~\mref{thm:gsbd0l}). Next we obtain a new linear basis of the free differential algebra of nonzero weight (Theorem~\mref{thm:gsbd}). Finally for the case of modified differential algebras, we errect two new linear bases of the free weighted modified differential algebra (Theorem~\mref{thm:gsbmd0r} and Theorem~\mref{thm:gsbmd0l}).
Parallel to the case of differential algebra,
some new linear bases of the free nonunitary (modified) Rota-Baxter algebra were constructed in Section 4,
utilizing the monomial orders in the last section.
After explaining why we narrow ourself to the nonunitary case, we obtain two new linear bases of the free nonunitary Rota-Baxter algebra (Theorem~\mref{thm:rbgs} and Theorem~\mref{thm:rbgs2}).
Similarly, these can be done for the free nonunitary modified Rota-Baxter algebra (Theorem~\mref{thm:mrbgs} and Theorem~\mref{thm:mrbgs2}).

\section{Free operated algebras and Gr\"{o}bner-Shirshov bases}
In this section, we review the construction of free operated algebras and
some required terminologies of the theory of Gr\"{o}bner-Shirshov bases.

\subsection{Free operated algebras} Let us start with the concept of operated algebras.
\begin{defn}~\mcite{BCQ, Gop}
An {\bf operated monoid} (resp. {\bf operated algebra}) is a monoid (resp. algebra)
$U$ together with a map (resp. linear map) $P_U: U\to U$.
\mlabel{de:mapset}
\end{defn}

Let $X$ be a set. Denote by $M(X)$ the free monoid on $X$ with the identity $\bfone$.
We will construct the free operated monoid $\frakM(X)$ on $X$.
The construction proceeds via the stages $\frakM_n(X)$ recursively defined as follows.
The initial stage is $\frakM_0(X) := M(X)$ and $\frakM_1(X) := M(X \cup \lc \frakM_0(X)\rc)$,
where $\lc \frakM_0(X)\rc:= \{ \lc u\rc \mid u\in \frakM_0(X)\}$ is a disjoint copy of $\frakM_0(X)$.
The inclusion $X\hookrightarrow X \cup \lc \frakM_0\rc $ induces a monomorphism
$$i_{0}:\mapmonoid_0(X) = M(X) \hookrightarrow \mapmonoid_1(X) = M(X \cup \lc \frakM_0\rc  )$$
of monoids through which we identify $\mapmonoid_0(X) $ with its image in $\mapmonoid_1(X)$.

For~$n\geq 2$, assume inductively that
$\frakM_{ n-1}(X)$ has been defined and the embedding
$$i_{n-2,n-1}\colon  \frakM_{ n-2}(X) \hookrightarrow \frakM_{ n-1}(X)$$
has been obtained. Then we define
\begin{equation*}
 \label{eq:frakn}
 \frakM_{ n}(X) := M \big( X\cup\lc\frakM_{n-1}(X) \rc\big).
\end{equation*}
Since~$\frakM_{ n-1}(X) = M \big(X\cup \lc\frakM_{ n-2}(X) \rc\big)$ is a free monoid,
the injection
$$  \lc\frakM_{ n-2}(X) \rc \hookrightarrow
    \lc \frakM_{ n-1}(X) \rc $$
induces a monoid embedding
\begin{equation*}
    \frakM_{ n-1}(X) = M \big( X\cup \lc\frakM_{n-2}(X) \rc \big)
 \hookrightarrow
    \frakM_{ n}(X) = M\big( X\cup\lc\frakM_{n-1}(X) \rc \big).
\end{equation*}
Finally we define the monoid
$$ \frakM (X):=\bigcup_{n\geq 0}\frakM_{ n}(X),$$
whose elements are called {\bf bracketed words} or
{\bf  bracketed monomials on $X$}. The identity of $\frakM(X)$ is $\bfone$.

Let $\kdot\frakM(X)$ be the free module spanned by
$\frakM(X)$. The multiplication on $\frakM(X)$ extends by linearity to turn the
module $\kdot\frakM(X)$ into an algebra.
Furthermore, we extend the operator
$\lc\ \rc: \frakM(X) \to \frakM(X), \ w\mapsto \lc w\rc$
to an operator $P$ on $\bfk\frakM(X)$ by linearity,
turning the algebra $\bfk\frakM(X)$ into an operated algebra.

\begin{lemma}{\rm (\cite[Corollary~3.7]{Gop})}
Let $i_X:X \to \frakM(X)$ and $j_X: \frakM(X) \to \bfk\mapm{X}$ be the natural embeddings. Then, with the notations above,
\begin{enumerate}
\item
the triple $(\frakM(X),P, i_X)$ is the free operated monoid on $X$; and
\label{it:mapsetm}
\item
the triple $(\bfk\mapm{X},P, j_X\circ i_X)$ is the free operated algebra
on $X$. \label{it:mapalgsg}
\end{enumerate}
\mlabel{pp:freetm}
\end{lemma}

\subsection{Gr\"{o}bner-Shirshov bases} \mlabel{ssec:gsbasis}
In this subsection, we provide some backgrounds on Gr\"obner-Shirshov bases~\cite{BCQ, GGZ, GSZ}.

\begin{defn}
Let $\star$ be a symbol not in $X$ and $X^\star = X \sqcup \{\star\}$.
\begin{enumerate}
\item By a {\bf $\star$-bracketed word} on $X$, we mean any bracketed word in $\mapm{X^\star}$ with exactly one occurrence of $\star$, counting multiplicities. The set of all $\star$-bracketed words on $X$ is denoted by $\frakM^{\star}(X)$.
\item For $q\in \mstar$ and $u \in  \frakM({X})$, we define $\stars{q}{u}:= q|_{\star \mapsto u}$ to be the bracketed word on $X$ obtained by replacing the symbol $\star$ in $q$ by $u$.

\item For $q\in \mstar$ and $s=\sum_i c_i u_i\in \bfk\frakM{(X)}$, where $c_i\in\bfk$ and $u_i\in \frakM{(X)}$, we define
\begin{equation*}
 q|_{s}:=\sum_i c_i q|_{u_i}\,. \vspace{-5pt}
\end{equation*}
\end{enumerate}
\mlabel{def:starbw}
\end{defn}

Operated ideals in $\bfk\mapm{X}$ can be characterized by $\star$-bracketed words.

\begin{lemma}{\rm (\cite[Lemma~3.2]{GSZ})}
Let $S \subseteq \bfk\mapm{X}$. Then the operated ideal generated by $S$ is
\begin{equation*}
\hspace{10pt}\Id(S) = \left\{\, \sum_{i} c_i q_i|_{s_i} \mid c_i\in \bfk,
q_i\in \frakM^{\star}(X), s_i\in S \right\}.
\end{equation*}
\mlabel{lem:opideal}
\end{lemma}

A monomial order is a well-order compatible with all operations, and the definition of Gr\"{o}bner-Shirshov bases is based on a fixed monomial order.

\begin{defn}
A {\bf monomial order on $\frakM(X)$} is a
well-order $\leq$ on $\frakM(X)$ such that
\begin{equation}
u < v \Longrightarrow q|_u < q|_v, \,\text{ for } \, u, v \in \frakM(X)\,\text{ and }\, q \in \frakM^{\star}(X).
\mlabel{eq:morder}
\end{equation}
\mlabel{de:morder}
\end{defn}

Since $\leq$ is a well-order, it follows from Eq.~\meqref{eq:morder} that $1 \leq u$ and $u < \lc u \rc$ for all $u \in \frakM(X)$.

\begin{defn}
Let $s \in \bfk\mapm{X}$ and $\leq$ a linear order on $\frakM(X)$.
\begin{enumerate}
\item Let $s \notin \bfk$. The {\bf leading monomial} of $s$, denoted by $\lbar{s}$, is the largest monomial appearing in $s$. The {\bf leading coefficient of $s$}, denoted by $c_s$, is the coefficient of $\lbar{s}$ in $s$.

\label{item:monic}

\item If $s \in \bfk$, we define the {\bf leading monomial of $s$} to be $1$ and the {\bf leading coefficient of $s$} to be $c_s= s$.\label{item:scalar}

\item The element $s$ is called {\bf monic  with respect to $\leq$} if $s\notin \bfk $ and $c_s=1$. A subset $S \subseteq \bfk\mapm{X}$ is called {\bf monic with respect to $\leq$} if every $s \in S$ is monic with respect to $\leq$.
\mlabel{it:res}
\end{enumerate}
\mlabel{def:irrS}
\end{defn}

The following are the main concepts of \gsbs based on $\bfk\frakM(X)$.
A nonunit element $w$ of $\mapm{X}$ can be uniquely expressed in the form
$$w=w_1w_2\cdots w_k$$
for some $k$ and some $w_i \in X \cup \blw{\mapm{X}}$, $1\leq i\leq k.$ In this case, we call $w_i$ {\bf prime} and $k$ the {\bf breadth} of $w$, denoted by $|w|:=k$. We also define $|\bfone| = 0$ for convenience.

\begin{defn}
Let $\leq$ be a monomial order on $\mapm{X}$ and $f, g \in\bfk\mapm{X}$ monic.
\begin{enumerate}
  \item  If there are $u, v,w\in \mapm{X}$ such that $w = \lbar{f}u = v\lbar{g}$ with max$\{|\lbar{f}|, |\lbar{g}|\} < |w| <
|\lbar{f}| + |\lbar{g}|$, we call
$$(f,g)_w:=(f,g)^{u,v}_w:= fu - vg$$
the {\bf intersection composition of $f$ and $g$ with respect to $w$}.
\mlabel{item:intcomp}
  \item  If there are $w\in\mapm{X}$ and $q\in\frakM^{\star}(X)$ such that $w = \lbar{f} = q\suba{\lbar{ g}},$ we call
$$(f,g)_w:=(f,g)^q_w := f - q\suba{g}$$
the {\bf including composition of $f$ and $g$ with respect to $w$}.
\mlabel{item:inccomp}
\end{enumerate}
\mlabel{defn:comp}
\end{defn}

\begin{defn}
Let $\leq$ be a monomial order on $\mapm{X}$ and  $S\subseteq\bfk\mapm{X}$ monic.
\begin{enumerate}
\item  An element $f\in\bfk\mapm{X}$ is called {\bf  trivial modulo $(S, w)$} with $w\in \frakM(X)$ if
$$f =\sum_ic_i q_i\suba{s_i} \, \text{ with } \lbar{q_i\suba{s_i}} < w, \,\text{ where } c_i\in\bfk, q_i\in\frakM^{\star}(X), s_i\in S,$$
denoted by $f\equiv 0 \mod(S, w)$. We write $f \equiv g \mod(S, w)$ if $f-g \equiv 0 \mod(S, w)$.

\item   We call $S$ a {\bf \gsb} in $\bfk\mapm{X}$ with respect
to $\leq$ if, for all pairs $f, g \in S$, every intersection composition of the form $(f, g)^{u,v}_w$
is trivial modulo $(S, w)$, and every including composition of the form $(f, g)^q_w$ is trivial modulo $(S, w)$.
\end{enumerate}
\end{defn}

The Composition-Diamond lemma is the corner stone of the theory of Gr\"{o}bner-Shirshov bases.

\begin{lemma} {\rm (Composition-Diamond lemma \cite{GSZ})}
Let $\leq$ be a monomial order on $\mapm{X}$ and $S \subseteq \bfk\mapm{X}$ monic with respect to $\leq$.
Then the following conditions are equivalent.
\begin{enumerate}
\item $S $ is a Gr\"{o}bner-Shirshov basis in $\bfk\mapm{X}$.
\label{it:cd1}

\item For every nonzero $f \in \Id(S)$, $\lbar{f}=q|_{\overline{s}}$
for some $q \in \frakM^\star(X)$ and some $s\in S$. \label{it:cd2}

\item  For every nonzero $f \in \Id(S)$, $f$ can be expressed in triangular form
\begin{equation*}
f= c_1\stars{q_1}{s_1}+ c_2\stars{q_2}{s_2}+\cdots+ c_nq_n|_{s_n}, \label{eq:fexp1}
\end{equation*}
where $0\neq c_i\in \bfk$, $s_i\in S$, $q_i\in \frakM^{\star}(X)$ for $1 \leq i \leq n$, and
$$ \stars{q_1}{\overline{s_1}}> \stars{q_2}{\overline{s_2}}
> \cdots> q_n|_{\overline{s_n}}.$$ \label{it:cd3}

\item $\eta(\Irr(S))$ is a $\bfk$-basis of $\bfk\mapm{X}/\Id(S)$, where
$\eta\!: \bfk\mapm{X} \to \bfk\mapm{X}/\Id(S)$ is the canonical homomorphism of modules and
\begin{equation*}
\irr(S)= \frakM(X) \setminus \{ q|_{\lbar{s}} \mid s\in S, q\in\frakM^{\star}(X)\}.
\end{equation*}
More precisely as modules, $$\bfk \mapm{X}=\bfk\Irr(S)\oplus \Id(S).$$
\label{it:cd4}
\end{enumerate}
\mlabel{lem:cdl}
\end{lemma}

The following is the nonunitary version of the above lemma.
The key point is in place of the free monoid in the construction of $\mapm{X}$
by the free semigroup to obtain $\frakS(X)$. See\mcite{BCQ} for more details.

\begin{lemma} {\rm (Composition-Diamond lemma \cite{BCQ})}
Let $\leq$ be a monomial order on $\frakS(X)$ and $S \subseteq \bfk\frakS(X)$ monic with respect to $\leq$.
Then the following conditions are equivalent.
\begin{enumerate}
\item $S $ is a Gr\"{o}bner-Shirshov basis in $\bfk\frakS(X)$.
\label{it:cdnu1}

\item For every nonzero $f \in \Id(S)$, $\lbar{f}=q|_{\overline{s}}$
for some $q \in \frakS^\star(X)$ and some $s\in S$. \label{it:cdnu2}

\item  For every nonzero $f \in \Id(S)$, $f$ can be expressed in triangular form
\begin{equation*}
f= c_1\stars{q_1}{s_1}+ c_2\stars{q_2}{s_2}+\cdots+ c_nq_n|_{s_n}, \label{eq:fexpnu1}
\end{equation*}
where $0\neq c_i\in \bfk$, $s_i\in S$, $q_i\in \frakS^{\star}(X)$ for $1 \leq i \leq n$, and
$$ \stars{q_1}{\overline{s_1}}> \stars{q_2}{\overline{s_2}}
> \cdots> q_n|_{\overline{s_n}}.$$ \label{it:cdnu3}

\item $\eta(\Irr(S))$ is a $\bfk$-basis of $\bfk\frakS(X)/\Id(S)$, where
$\eta\!: \bfk\frakS(X) \to \bfk\frakS(X)/\Id(S)$ is the canonical homomorphism of modules and
\begin{equation*}
\irr(S)= \frakS(X) \setminus \{ q|_{\lbar{s}} \mid s\in S, q\in\frakS^{\star}(X)\}.
\end{equation*}
More precisely as modules, $$\bfk \frakS(X)=\bfk\Irr(S)\oplus \Id(S).$$
 \label{it:cdnu4}
\end{enumerate}
\mlabel{lem:cdlnu}
\end{lemma}

\section{New bases of free weighted (modified) differential algebras}
In this section, we obtain some new linear bases of free weighted differential algebras and free weighted modified differential
algebras, in terms of the method of Gr\"{o}bner-Shirshov bases.

\subsection{Two monomial orders on $\mapm{X}$}
We are going to build two new monomial orders on top of $\mapm{X}$,
which are motivated from the path-lexicographical order on planar rooted trees~\cite{He}.
Let us first give an example to expose our idea. Denote $\X= X\sqcup\{\lc \bfone\rc\}$.

\begin{exam} Let $w=\blw{x\blw{y}z}\lc \bfone\rc\in\mapm{X}$ with $x,y,z\in X$.
From left to right, $w$ has four letters $x, y, z,\lc \bfone\rc\in \X$.
We would like to define $\pat_w(x):= \mu P\mu x$ for the first letter $x$, $\pat_w(y):= \mu P\mu Py$ for the second letter $y$,
$\pat_w(z):= \mu P \mu z$ for the third letter $z$ and $\pat_w(\blw{\bfone}):= \mu P$ for the fourth letter $\lc \bfone\rc$.
Here $P$ and $\mu$ stand for the operator $\lc \ \rc$
and the multiplication, respectively. Finally, define
$$\patl(w) := \Big(\pat_w(x),\pat_w(y), \pat_w(z), \pat_w(\lc \bfone\rc)\Big)=( \mu P\mu x,\mu P\mu Py,\mu P \mu z, \mu P )$$
and
$$\patr(w):= \Big(\pat_w(\lc \bfone\rc), \pat_w(z), \pat_w(y), \pat_w(x)\Big)=( \mu P , \mu P \mu z, \mu P\mu Py, \mu P\mu x).$$
\hspace*{\fill}$\square$
\end{exam}

Formally, let $X$ be a set of letters and $\X = X\sqcup\{\lc\bfone\rc\}$.
Define $\degx(\bfone):= 0$. For $w\in \frakM(X)\setminus \{\bfone\}$, denote by $\degx(w)$ the total number of all occurrences of all letters $x\in \X$ in $w$. For example,
$$\degx(\lc \bfone\rc) = 1, \, \degx(x)=1,\,\degx(xx)=2\, \text{ and }\, \degx(x\lc \bfone\rc) = 2.$$

Let $P$ and $\mu$ be two symbols not in $\X$, representing the operator $\lc \ \rc$
and the multiplication, respectively.
Let $w\in \frakM(X)$. Recall $|w|$ is the breadth of $w$ and $\dep(w) := {\rm min}\{n \mid w\in\frakM_n(X)\}$ is the depth of $w$.
Assume that $x_1, \cdots, x_k$ are all letters in $\X$ appearing in $w$, from left to right.
Notice that $x_i$ and $x_j$ may be equal for different $i$ and $j$. For each $x_i$ with $1\leq i\leq k$, we define
$\pat_w(x_i)\in M(X\sqcup \{ P,\mu\})$ inductively on $\dep(w)\geq 0$. For the initial step of $\dep(w)=0$, if $w = \bfone$, then
define $\pat_w(x_i):= \bfone$; if $w\neq \bfone$, then $x_i\in X$ and define
$$\pat_w(x_i):=
\left\{
\begin{array}{ll}
x_i\in M(X\sqcup \{P, \mu\}), & \hbox{if $|w|=1$;}\\
\mu x_i \in M(X\sqcup \{P, \mu\}), & \hbox{if $|w|\geq2 $.}\\
\end{array}
\right.$$
For $\dep(w)\geq1$, we reduce to the induction on breadth $|w|\geq 1$.
If $w = \lc \bfone\rc$, then $x_i = \lc \bfone\rc$ and define $\pat_w(x_i):= P.$
Suppose $w\neq \lc \bfone\rc$. If $|w|=1$, then $w=\blw{w'}$ for some $w'\in \mapm{X}\setminus\{\bfone\}$ and define
$$\pat_w(x_i):= P\,\pat_{w'}(x_i)\in M(X\sqcup \{P, \mu\} ),$$
where $\pat_{w'}(x_i)$ is defined by the induction hypothesis on depth.
If $|w|\geq2$, then $w=w_1\cdots w_n$ for some prime elements $w_1, \cdots, w_n$ and
$x_i$ appears in some $w_j$. Define
$$\pat_w(x_i):=\mu \pat_{w_j}(x_i),$$
where $\pat_{w_j}(x_i)$ is defined by the induction hypothesis on breadth.
Finally, we define
\begin{equation*}
\begin{split}
\patl(w):=&\Big(\pat_w(x_1),\pat_w(x_2), \cdots, \pat_w(x_k)\Big)\in (M(X\sqcup \{P, \mu\}))^k,\\
\patr(w):=&\Big(\pat_w(x_k), \pat_w(x_{k-1}),\cdots, \pat_w(x_1)\Big)\in (M(X\sqcup \{P, \mu\}))^k.
\end{split}
\end{equation*}
In particular, we have
\begin{align*}
\patl(\bfone)=&\ \patr(\bfone)=\bfone,\\
\patl(\blw{\bfone})=&\ \patr(\blw{\bfone})= P = P\bfone =P\, \patl(\bfone),\\
\patl(\blw{\bfone}\blw{\bfone})=&\ (\mu P,\mu P) =\patr(\blw{\bfone}\blw{\bfone}).
\end{align*}

We give an example for better understanding.

\begin{exam}
Let $u=\blw{x\blw{y}}\blw{\bfone} \in\mapm{X}$ and $v=\blw{xy}z$ with $x,y,z\in X$. Then
\begin{align*}
\patl(u)=&~\Big(\pat_u(x),\pat_u(y),\pat_u(\lc\bfone\rc) \Big)= (\mu P\mu x,\mu P\mu Py,\mu P),\\
\patr(u)=&~\Big(\pat_u(\lc\bfone\rc),\pat_u(y),\pat_u(x) \Big)= (\mu P, \mu P\mu Py, \mu P\mu x),\\
\patl(v)=&~\Big(\pat_v(x),\pat_v(y),\pat_v(z) \Big)=(\mu P\mu x,\mu P\mu y,\mu z),\\
\patr(v)=&~\Big(\pat_v(z),\pat_v(y),\pat_v(x)\Big)=(\mu z, \mu P\mu y,\mu P\mu x).
\end{align*}
\qed
\end{exam}

Let $(X,\leq)$ be a well-ordered set.
Recall the {\bf degree lexicographical order} $\dlex $ on $M(X)$ is, for  any $u=u_1\cdots u_k\, ,v=v_1\cdots v_\ell \in M(X)\setminus\{\bfone\}$, where $u_1,\cdots,u_k,v_1,\cdots, v_\ell\in X$,
\begin{equation*}
u\dlex v\Leftrightarrow
\left\{
\begin{array}{l}
|u|<|v|,\\[5pt]
\text{or}~ |u|=|v|(=k) ~ \text{and} ~ (u_1, \cdots, u_k) \leq  (v_1, \cdots, v_k)\, \text{ lexicographically}.
\end{array}
\right.
\end{equation*}
Here we use the convention that $\bfone\leq  u$ for all $u\in M(X)$.
It is well-known that the degree lexicographical order $\dlex$ on $M(X)$ is a well-order~\mcite{BN}.
Assume $x<P<\mu$ for all $x\in X$. Then $X\sqcup \{P, \mu\}$ is a well-ordered set,
and we have a degree lexicographical order $\dlex$ on $M(X\sqcup \{P, \mu\})$.
Let $u,v \in\mapm{X}$. Define
\begin{align}
u\dbl v \Leftrightarrow& \Big(\degx(u), \patl(u)\Big)\leq\Big(\degx(v), \patl(v)\Big)\,  \text{ lexicographically},\mlabel{eq:ordl}\\
u\dbr v \Leftrightarrow& \Big(\degx(u), \patr(u)\Big)\leq\Big(\degx(v), \patr(v)\Big)\,  \text{ lexicographically}.
\mlabel{eq:ordr}
\end{align}
Here $\pat_u(x_i)$ and $\pat_v(y_j)$ in $M(X\sqcup \{P, \mu\})$ are compared by $\dlex$.
Let us expose an example.

\begin{exam}
Let $u=\blw{x}y$, $v=x\blw{y}$ and $w=\blw{xy}$ with $x,y \in X$.  Then
\begin{align*}
 \patl(u)&=\Big(\mu Px, \mu y\Big),\quad \patr(u)=\Big( \mu y, \mu Px\Big),\\
\patl(v)&=\Big(\mu x, \mu Py\Big),\quad\patr(v)=\Big( \mu Py, \mu x\Big),\\
 \patl(w)&=\Big(P\mu x, P\mu y\Big),\, \patr(w)=\Big( P\mu y, P\mu x\Big).
\end{align*}
Since $$\degx(u)=\degx(v)=\degx(w)=2\,\text{ and }\, \mu P x >_{{\rm dlex}} P\mu x >_{{\rm dlex}} \mu x$$
by $\mu > P$, we have $u\dbln w\dbln v.$ Similarly, $ v\dbrn w\dbrn u.$
\hspace*{\fill}$\square$
\end{exam}

\begin{remark}
In the above constructions of $\dbl$ and $\dbr$, the assumption $\mu>P>x$ for all $x\in X$ is to insure the required leading monomials used
for the case of weighted (modified) differential algebras in Section 3 and the case of (modified) Rota-Baxter algebras in Section 4.
\end{remark}

By the above construction process, the two orders $\dbl$ and $\dbr$ are indeed well-orders.
To prove this, the following result is needed.
Let $\leq_{Y_i}$ be a well-order on $Y_i$ for $1\leq i\leq k$ and $k\geq 1$.
Recall that the lexicographical order $\leq_{\clex}$ on the cartesian product $Y_1  \times Y_2\times\cdots \times Y_k$ is defined recursively by
\begin{eqnarray}
(x_1,x_2,\cdots,x_k) &\leq_{\clex}&(y_1,y_2, \cdots,y_k) \nonumber\\&\Leftrightarrow&
\left\{\begin{array}{ll} x_1<_{Y_1}y_1,\\[5pt]
\text{or}~ x_1= y_1 ~\text{and}~ (x_2,\cdots,x_k)\leq_{\clex}(y_2, \cdots,y_k),
\end{array}\right.\nonumber
\end{eqnarray}
where $(x_2,\cdots,x_k)\leq_{\clex} (y_2,\cdots,y_k)$ is defined by the induction hypothesis, with the convention that $\leq_{\clex}$ is the trivial relation when $k=1$.

\begin{lemma}\mcite{Ha}
Let $\leq_{Y_i}$ be a well-order on $Y_i$ for $1\leq i\leq k$ and $k\geq 1$.
Then the lexicographical order $\leq_{\clex}$ is a well-order on the cartesian product
$Y_1 \times Y_2\times\cdots \times Y_k$.
\mlabel{lem:carw}
\end{lemma}

As a preparation, we expose the following result.

\begin{lemma}
With the above setting, the orders $\dbl$ and $\dbr$ defined in Eqs.~(\mref{eq:ordl}) and~(\mref{eq:ordr}) are well-orders on $\mapm{X}$.
\mlabel{lem:tworder}
\end{lemma}

\begin{proof}
We only consider the case of $\dbl$, as the case of $\dbr$ is similar.
Since $\dbl$ is a linear order on $\frakM(X)$ by Eq.~(\mref{eq:ordl}), it suffices to
verify that $\dbl$ satisfies the descending chain condition.

Let $$u_1 \geq_{{\rm PL_{{\rm l}}}} u_2 \geq_{{\rm PL_{{\rm l}}}} \cdots $$
be a given descending chain in $\frakM(X)$. Note that $\deg_{\X}(u_i)\in \ZZ_{\geq 0}$ for $i\geq 1$
and $\ZZ_{\geq 0}$ equipped with the usual order is a well-ordered set. There is an $m\geq 1$
such that $$k:= \deg_{\X}(u_m) =\deg_{\X}(u_{m+1}) =\cdots.$$
Then $\patl(u_i) \in ( M(X\sqcup \{P, \mu \}))^k$ for $i\geq m$.
Since the $(M(X\sqcup \{P, \mu \}), \dlex)$ is a well-ordered set, the $( M(X\sqcup \{P, \mu \}))^k$
is also a well-ordered set by Lemma~\mref{lem:carw}. Thus there is an $n\geq m$ such that
$$\patl(u_n) = \patl(u_{n+1}) = \cdots,$$
and so $$u_n =_{{\rm PL_{{\rm l}}}} =_{{\rm PL_{{\rm l}}}} u_{n+1} =_{{\rm PL_{{\rm l}}}} \cdots.$$
This completes the proof.
\end{proof}

Further we are going to show that $\dbl$ and $\dbr$ are monomial orders.
A well-order $\leq$ on $\mapm{X}$ is called {\bf bracket compatible} (resp. {\bf left
(multiplication) compatible}, resp. {\bf right (multiplication) compatible}) if
$u\leq v$ implies $\blw{u}\leq \blw{v}$, (resp. $wu\leq wv$, resp. $uw\leq vw$, for all $w\in\mapm{X}$).

\begin{lemma}~\mcite{ZGGS}
A well-order $\leq$ is a monomial order if and only if it is bracket compatible, left compatible and right compatible.
\mlabel{lem:ord}
\end{lemma}

Now we are ready for the main result on orders in this section.

\begin{theorem}
With the above setting, the orders $\dbl$ and $\dbr$ given in Eqs.~(\mref{eq:ordl}) and~(\mref{eq:ordr}) are monomial orders on $\mapm{X}$.
\mlabel{thm:rbord}
\end{theorem}

\begin{proof}
By Lemmas~\mref{lem:tworder} and~\mref{lem:ord}, we are left to show that $\dbl$ and $\dbr$ are bracket compatible, left compatible and right compatible. For this we only consider the case of $\dbl$, as the case of $\dbr$ is similar.
Let $u\dbl v$ with $u,v\in \frakM(X)$  and $w\in \frakM(X)$.

{\bf(Bracket compatible)}
By the definition of $\dbl$, we have
\begin{equation}
\begin{aligned}
\Big(\degx(u), \patl(u)\Big)\leq\Big(\degx(v), \patl(v)\Big)\,  \text{ lexicographically}.
\mlabel{eq:uv}
\end{aligned}
\end{equation}
Further we have
\begin{equation*}
\Big(\degx(\blw{u}), \patl(\blw{u})\Big)=\left\{
    \begin{array}{ll}
      \Big(\degx({u})+1, P\,\patl({u})\Big)=(1, P), & \hbox{if $u=\bfone$;} \\
       \Big(\degx({u}), P\,\patl({u})\Big), & \hbox{if $u\neq \bfone$.}
     \end{array}
 \right.
\end{equation*}
and
\begin{equation*}
\Big(\degx(\blw{v}), \patl(\blw{v})\Big)=\left\{
    \begin{array}{ll}
     \Big(\degx({v})+1, P\,\patl({u})\Big)=(1, P), & \hbox{if $v=\bfone$;} \\
       \Big(\degx({v}), P\,\patl({v})\Big), & \hbox{if $v\neq \bfone$.}
     \end{array}
 \right.
\end{equation*}
By the assumption $u\dbl v$, we have the following three cases to consider.

\noindent{\bf Case 1.} If  $u=v=\bfone$, then $\blw{u}=\blw{v}=\blw{\bfone}$ and the result holds.

\noindent{\bf Case 2.} If  $u=\bfone$ and $v\neq\bfone$, then $\degx({\lc v\rc}) = \degx({v})\geq1$ and so
\begin{align*}
 \Big(\degx(\blw{u}), \patl(\blw{u})\Big)=(1, P)<\Big(\degx(\blw{v}), \patl(\blw{v})\Big)\,\text{ lexicographically}.
\end{align*}
Thus $\blw{u}\dbl \blw{v}$.

\noindent{\bf Case 3.} If $u\neq\bfone\,\text{ and }\,v\neq\bfone$, then by Eq.~\meqref{eq:uv},
\begin{align*}
&\Big(\degx(\blw{u}), \patl(\blw{u})\Big)=\Big(\degx({u}), P\,\patl({u})\Big)\\
\leq&\Big(\degx({v}), P\,\patl({v})\Big)= \Big(\degx(\blw{v}), \patl(\blw{v})\Big)
\,\text{ lexicographically},
\end{align*}
and so $\blw{u}\dbl \blw{v}$.

{\bf (Left compatible)} By Eq.~\meqref{eq:uv},  we have
\begin{align*}
&\Big(\degx(wu), \patl(wu)\Big)=\Big(\degx(w)+\degx(u), \mu\patl(w),\mu\patl(u)\Big)\\
\leq&\ \Big(\degx(w)+\degx(v), \mu\patl(w),\mu\patl(v)\Big) = \Big(\degx(wv), \patl(wv)\Big)\,\text{ lexicographically},
\end{align*}
so $wu \dbl wv.$

{\bf (Right compatible)} It follows from Eq.~\meqref{eq:uv} that
\begin{align*}
&\Big(\degx( uw), \patl(uw)\Big)=\Big(\degx(u)+\degx(w), \mu\patl(u),\mu\patl(w)\Big)\\
\leq&\ \Big(\degx(v)+\degx(w), \mu\patl(v),\mu\patl(w)\Big)= \Big(\degx( vw), \patl(vw)\Big)\,\text{ lexicographically},
\end{align*}
Thus $uw \dbl vw$. This completes the proof.
\end{proof}

\subsection{New bases of free weighted differential algebras}
The following is the concept of a weighted differential algebra, which is a generalization of the classical differential algebra and difference algebra.

\begin{defn}\cite{GK}
Let $\lambda\in \bfk$ be a fixed element. A {\bf differential algebra of weight $\lambda$} is an algebra $A$ together with a linear operator $D$: $A\longrightarrow A$ such that $D(1)=0$ and
\begin{equation}
D(xy)=D(x)y+xD(y)+\lambda D(x)D(y), \,\text{ for }\, x,y \in A.
\mlabel{eq:wdiff}
\end{equation}
\end{defn}

Now we arrive at our first main result in this section, concerning the case of weight zero.

\begin{theorem} Let $(X, \leq)$ be a well-ordered set and $\dbl$ the monomial order on $\frakM(X)$ defined in Eq.~\meqref{eq:ordl}.
\begin{enumerate}
\item The set
$$S_{\rm diff_{0, l}} := \Big\{\blw{x}y-\blw{xy}+x\blw{y}  \mid x,y\in \mapm{X}\Big\}$$
is a Gr\"{o}bner-Shirshov basis in $\bfk\mapm{X}$. \mlabel{it:diff01}

\item The set $${\rm Irr}(S_{\rm diff_{0, l}}):= \Big\{\frakM(X)\setminus \{q|_{\blw{x}y} \mid x,y\in \frakM(X),q \in \frakM^\star(X)\}\Big\}$$
is a \bfk-basis of the free differential algebra $\bfk\frakM(X)/\Id(S_{\rm diff_{0, l}})$ of weight zero. \mlabel{it:diff02}
\end{enumerate}
\mlabel{thm:gsbd0}
\end{theorem}

\begin{proof}
\mref{it:diff01} Note that we can replace $S_{\rm diff_{0, l}}$ by the following set
$$ S_{\rm diff_{0, l}} = \Big\{\blw{x}y-\blw{xy}+x\blw{y} , \blw{\bfone} \mid x,y\in \mapm{X}\setminus \{{\bfone}\}\Big\}.$$
Since compositions involving $\blw{\bfone}$ are trivial, it suffices to consider
compositions not involving $\blw{\bfone}$.
Denote by
$$f(x,y) = \blw{x}y-\blw{xy}+x\blw{y}\in S_{\rm diff_{0, l}}, \,\text{ where }\, x,y\in \mapm{X}\setminus \{{\bfone}\}.$$
By the definition of $\dbl$, we have the leading monomial
$$\lbar{f(x,y)} = \blw{x}y.$$
The ambiguities of all possible intersection compositions are
$$w_1=\blw{x}\blw{y}z, \text{ where } x,y,z \in \mapm{X}\setminus \{{\bfone}\}, $$
and the ambiguities of all possible including compositions are
$$w_2=\blw{q\suba{\blw{x}y}}z\,\text{ and }\, w_3=\blw{x}q\suba{\blw{y}z},  \text{ where } x,y,z \in \mapm{X}\setminus \{{\bfone}\}\,\text{ and }\, q \in \frakM^\star(X).$$
For the ambiguity $w_1$, it is from an intersection composition
$$\Big(f(x,\blw{y}), f(y,z)\Big)_{w_1},$$
which is trivial by
\begin{align*}
&\Big(f(x,\blw{y}), f(y,z)\Big)_{w_1}\\
=&~ \big(\blw{x}\blw{y}-\blw{x\blw{y}}+x\blw{\blw{y}}\big)z-\blw{x}\big(\blw{y}z-\blw{yz}+y\blw{z}\big)\\
=&~-\blw{x\blw{y}}z+x\blw{\blw{y}}z+\blw{x}\blw{yz}-\blw{x}y\blw{z}\\
\equiv&~ -\blw{x\blw{y}z}+x\blw{y}\blw{z}+x\blw{\blw{y}z}-x\blw{y}\blw{z}+\blw{x\blw{yz}}-x\blw{\blw{yz}}-\blw{xy\blw{z}}+x\blw{y\blw{z}}\\
\equiv&~ -\blw{x\blw{y}z}+x\blw{\blw{y}z}+\blw{x\blw{yz}}-x\blw{\blw{yz}}-\blw{xy\blw{z}}+x\blw{y\blw{z}}\\
\equiv&~ \blw{-x\blw{yz}+xy\blw{z}}
+x\blw{\blw{yz}-y\blw{z}}
+\blw{x\blw{yz}}-x\blw{\blw{yz}}-\blw{xy\blw{z}}+x\blw{y\blw{z}}\\
\equiv&~0 \mod(S_{\rm diff_{0, l}}, w_1).
\end{align*}

Further the ambiguity $w_2$ is from an including composition
$$\Big(f(q\suba{\blw{x}y},z), f(x,y)\Big)_{w_2}.$$
Its triviality follows from
\begin{align*}
&\Big(f(q\suba{\blw{x}y},z), f(x,y)\Big)_{w_2}\\
=&~\Big(\blw{q\suba{\blw{x}y}}z-\blw{q\suba{\blw{x}y}z}+q\suba{\blw{x}y}\blw{z}\Big)
-\Big(\blw{q\suba{\blw{x}y-\blw{xy}+x\blw{y}}}z\Big)\\
=&~-\blw{q\suba{\blw{x}y}z}+q\suba{\blw{x}y}\blw{z}
+\blw{q\suba{\blw{xy}}}z-\blw{q\suba{x\blw{y}}}z\\
\equiv&~-\blw{q\suba{\blw{xy}-x\blw{y}}z}+q\suba{\blw{xy}-x\blw{y}}\blw{z}
+\blw{q\suba{\blw{xy}}z}-q\suba{\blw{xy}}\blw{z}-\blw{q\suba{x\blw{y}}z}+q\suba{x\blw{y}}\blw{z}\\
\equiv&~0 \mod(S_{\rm diff_{0, l}}, w_2).
\end{align*}

Finally the triviality of the composition from the ambiguity $w_3$ can be shown as
\begin{align*}
&\Big(f(x, q\suba{\blw{y}z}), f(y,z)\Big)_{w_3}\\
=&~\Big(\blw{x}q\suba{\blw{y}z}-\blw{xq\suba{\blw{y}z}}+x\blw{q\suba{\blw{y}z}}\Big)
-\Big(\blw{x}q\suba{\blw{y}z-\blw{yz}+y\blw{z}}\Big)\\
=&~-\blw{xq\suba{\blw{y}z}}+x\blw{q\suba{\blw{y}z}}
+\blw{x}q\suba{\blw{yz}}- \blw{x}q\suba{{y\blw{z}}}\\
\equiv&~-\blw{xq\suba{\blw{yz}-y\blw{z}}}+x\blw{q\suba{\blw{yz}-y\blw{z}}}
+\blw{xq\suba{\blw{yz}}}-x\blw{q\suba{\blw{yz}}}-\blw{xq\suba{{y\blw{z}}}}+x\blw{q\suba{{y\blw{z}}}}\\
\equiv&~0 \mod(S_{\rm diff_{0, l}}, w_3).
\end{align*}
Thus $S_{\rm diff_{0, l}}$ is a Gr\"{o}bner-Shirshov basis in $\bfk\mapm{X}$.

\mref{it:diff02} It follows from Lemma~\mref{lem:cdl} and Item~\mref{it:diff1}.
\end{proof}

Replacing the monomial $\dbl$ utilized in Theorem~\mref{thm:gsbd0} by $\dbr$, we have the following result.

\begin{theorem} Let $(X, \leq)$ be a well-ordered set and $\dbr$ the monomial order on $\frakM(X)$ defined in Eq.~\meqref{eq:ordr}.
\begin{enumerate}
\item The set
$$S_{\rm diff_{0, r}} := \Big\{x\blw{y} - \blw{xy}+ \blw{x} y  \mid x,y\in \mapm{X}\Big\}$$
is a Gr\"{o}bner-Shirshov basis in $\bfk\mapm{X}$. \mlabel{it:diff01}

\item The set $${\rm Irr}(S_{\rm diff_{0, r}}):= \Big\{\frakM(X)\setminus \{q|_{x\blw{y}} \mid x,y\in \frakM(X),q \in \frakM^\star(X)\}\Big\}$$
is a \bfk-basis of the free differential algebra $\bfk\frakM(X)/\Id(S_{\rm diff_{0, r}})$ of weight zero. \mlabel{it:diff02}
\end{enumerate}
\mlabel{thm:gsbd0l}
\end{theorem}

\begin{proof}
It is similar to the proof of Theorem~\mref{thm:gsbd0}.
\end{proof}

Next we move to the case of differential algebra of nonzero weight.

\begin{remark}
If $\lambda\neq 0$, the term-rewriting system
$$\{  \blw{x}y  \to \blw{xy}- x\blw{y} -\lambda \blw{x} \blw{y} \mid x,y \in\mapm{X}\}\,\text{ or }\,
\{ x\blw{y}  \to \blw{xy}-  \blw{x}y  -\lambda \blw{x} \blw{y} \mid x,y \in\mapm{X}\} $$
on $\bfk\frakM(X)$ is not terminating, since $\blw{x} \blw{y}$ on the right hand side can be rewritten again.
\mlabel{rk:notterm}
\end{remark}

Remark~\mref{rk:notterm} sheds a light on taking $\blw{x} \blw{y}$ as the leading monomial,
which can be realized by the monomial order $\dbl$ or $\dbr$.

\begin{theorem} Let $\lambda$ be a fixed invertible element in $\bfk$.
Let $X$ be a well-ordered set and $\dbl$ or $\dbr$  the monomial order on $\mapm{X}$  given in Eqs.~\meqref{eq:ordl} and~\meqref{eq:ordr}.
\begin{enumerate}
\item The set
$$S_{\rm diff_\lambda} = \Big\{\blw{x}\blw{y}-\frac{1}{\lambda}\blw{xy}+\frac{1}{\lambda}\blw{x}y+\frac{1}{\lambda}x\blw{y},\, \blw{\bfone} \mid x,y\in \mapm{X}\setminus \{\bfone\}\Big\}$$
is a Gr\"{o}bner-Shirshov basis in $\bfk\mapm{X}$. \mlabel{it:diff1}

\item The set
$${\rm Irr}(S_{\rm diff_\lambda}):= \Big\{\frakM(X)\setminus \{q|_{\blw{x}\blw{y}} , q|_{\blw{\bfone}}\mid x,y\in \frakM(X)\setminus \{\bfone \}, q \in \frakM^\star(X)\}\Big\}$$
 is a \bfk-basis of the free differential algebra $\bfk\frakM(X)/\Id(S_{\rm diff_\lambda})$ of weight $\lambda$. \mlabel{it:diff2}
\end{enumerate}
\mlabel{thm:gsbd}
\end{theorem}

\begin{proof}
\mref{it:diff1}
Notice that the composition involving $\blw{\bfone}$ are trivial.
It suffices to consider compositions not involving $\blw{\bfone}$.
Denote by
$$f(x,y)=\blw{x}\blw{y}-\frac{1}{\lambda}\blw{xy}+\frac{1}{\lambda}\blw{x}y+\frac{1}{\lambda}x\blw{y}\in S_{\rm diff_\lambda} .$$
For both $\dbl$ and $\dbr$, we have the leading monomial
$$\lbar{f(x,y)} = \blw{x}\blw{y}.$$
The ambiguities of all possible intersection compositions are
$$w_1=\blw{x}\blw{y}\blw{z},\text{ where }\, x,y,z \in \mapm{X}\setminus \{\bfone\},$$
and the ambiguities of all possible including compositions are
$$w_2=\blw{q\suba{\blw{x}\blw{y}}}\blw{z}\,\text{ and }\, w_3=\blw{x}\blw{q\suba{\blw{y}\blw{z}}}, \text{ where }\, x,y,z \in \mapm{X}\setminus \{\bfone\}\,\text{ and }\, q \in \frakM^\star(X).$$
The ambiguity $w_1$ is from an intersection composition
$$\Big(f(x,y), f(y,z)\Big)_{w_1},$$
whose triviality follows from
\begin{align*}
 &\Big(f(x,y), f(y,z)\Big)_{w_1}\\
=&~\Big(\blw{x}\blw{y}-\frac{1}{\lambda}\blw{xy}+\frac{1}{\lambda}\blw{x}y+\frac{1}{\lambda}x\blw{y}\Big)\blw{z}
-\blw{x}\Big(\blw{y}\blw{z}-\frac{1}{\lambda}\blw{yz}+\frac{1}{\lambda}\blw{y}z+\frac{1}{\lambda}y\blw{z}\Big)\\
=&~\frac{1}{\lambda}\Big(-\blw{xy}\blw{z}+x\blw{y}\blw{z}
+\blw{x}\blw{yz}-\blw{x}\blw{y}z\Big)\\
\equiv&~\frac{1}{\lambda}\Big(-\frac{1}{\lambda}\blw{xyz}+\frac{1}{\lambda}\blw{xy}z+\frac{1}{\lambda}xy\blw{z}   + \frac{1}{\lambda}x\blw{yz}-\frac{1}{\lambda}x\blw{y}z-\frac{1}{\lambda}xy\blw{z}\\
&+\frac{1}{\lambda}\blw{xyz}-\frac{1}{\lambda}\blw{x}yz-\frac{1}{\lambda}x\blw{yz}  -\frac{1}{\lambda}\blw{xy}z+\frac{1}{\lambda}\blw{x}yz+\frac{1}{\lambda}x\blw{y}z
\Big)\\
\equiv&~0 \mod(S_{\rm diff_\lambda} , w_1).
\end{align*}
The ambiguity $w_2$ is from an including composition
$$\Big(f(q\suba{\blw{x}\blw{y}},z), f(x,y)\Big)_{w_2},$$
which is trivial by
\begin{align*}
&\Big(f(q\suba{\blw{x}\blw{y}},z), f(x,y)\Big)_{w_2}\\
=&~\Big(\blw{q\suba{\blw{x}\blw{y}}}\blw{z}-\frac{1}{\lambda}\blw{q\suba{\blw{x}\blw{y}}z}+\frac{1}{\lambda}\blw{q\suba{\blw{x}\blw{y}}}z+\frac{1}{\lambda}q\suba{\blw{x}\blw{y}}\blw{z}\Big)
-\Big(\blw{q\suba{\blw{x}\blw{y}-\frac{1}{\lambda}\blw{xy}+\frac{1}{\lambda}\blw{x}y+\frac{1}{\lambda}x\blw{y}}}\blw{z}\Big)\\
=&~\frac{1}{\lambda}\Big(-\blw{q\suba{\blw{x}\blw{y}}z}+\blw{q\suba{\blw{x}\blw{y}}}z+q\suba{\blw{x}\blw{y}}\blw{z}
+\blw{q\suba{\blw{xy}}}\blw{z}-\blw{q\suba{\blw{x}{y}}}\blw{z}-\blw{q\suba{{x}\blw{y}}}\blw{z}\Big)\\
\equiv&~\frac{1}{\lambda}\Big(
\blw{q\suba{-\frac{1}{\lambda}\blw{xy}+\frac{1}{\lambda}\blw{x}y+\frac{1}{\lambda}x\blw{y}}z}
+\blw{q\suba{\frac{1}{\lambda}\blw{xy}-\frac{1}{\lambda}\blw{x}y-\frac{1}{\lambda}x\blw{y}}}z
+q\suba{\frac{1}{\lambda}\blw{xy}-\frac{1}{\lambda}\blw{x}y-\frac{1}{\lambda}x\blw{y}}\blw{z}\\
&+\frac{1}{\lambda}\blw{q\suba{\blw{xy}}z}-\frac{1}{\lambda}\blw{q\suba{\blw{xy}}}z-\frac{1}{\lambda}q\suba{\blw{xy}}\blw{z}
-\frac{1}{\lambda}\blw{q\suba{\blw{x}{y}}z}+\frac{1}{\lambda}\blw{q\suba{\blw{x}{y}}}z+\frac{1}{\lambda}q\suba{\blw{x}{y}}\blw{z}\\
&-\frac{1}{\lambda}\blw{q\suba{{x}\blw{y}}z}+\frac{1}{\lambda}\blw{q\suba{{x}\blw{y}}}z+\frac{1}{\lambda}q\suba{{x}\blw{y}}\blw{z}\Big)\\
\equiv&~0 \mod(S_{\rm diff_\lambda} , w_2).
\end{align*}
The case of the ambiguity $w_3$ is similar to the case of $w_2$.
Thus $S_{\rm diff_\lambda}$ is a Gr\"{o}bner-Shirshov basis in $\bfk\mapm{X}$.

\mref{it:diff2} Notice that Eq.~(\mref{eq:wdiff}) is trivial if $x=1$ or $y=1$.
Then the result follows from Lemma~\mref{lem:cdl} and Item~\mref{it:diff1}.
\end{proof}

\subsection{New bases of free weighted modified differential algebras}
The following is the concept of a weighted modified differential algebra.

\begin{defn}\mcite{PZGL}
Let $\lambda\in \bfk$ be a fixed constant. A {\bf modified differential algebra of weight $\lambda$} is an algebra $A$ together with a linear operator $D$: $A\longrightarrow A$ such that
\begin{equation}
D(xy)=D(x)y+xD(y)+\lambda xy, \,\text{ for }\, x,y \in A.
\mlabel{eq:mdiff}
\end{equation}
\end{defn}

\begin{remark}
\begin{enumerate}
\item It follows from Eq.~(\mref{eq:mdiff}) that $D(1) = -\lambda$.

\item The set
\begin{equation}
\{\blw{xy}- x\blw{y} -  \blw{x} y -\lambda xy \mid x,y\in \mapm{X}\}
\mlabel{eq:mdiff0}
 \end{equation}
is a Gr\"{o}bner-Shirshov basis in $\bfk\mapm{X}$~\cite[Lemma 5.3]{GSZ} with respect to the leading monomial $\lc xy\rc$.
\end{enumerate}
\end{remark}

We next prove that the set in Eq.~(\mref{eq:mdiff0})
with respect to the leading monomial $\blw{x}{y}$ or $x\blw{y}$ is also a Gr\"{o}bner-Shirshov basis in $\bfk\mapm{X}$.

\begin{theorem}
Let $(X, \leq)$ be a well-ordered set and $\dbl$ the monomial order on $\frakM(X)$ defined in Eq.~\meqref{eq:ordl}.
\begin{enumerate}
\item The set
$$S_{\rm mdiffl} := \Big\{\blw{x}y-\blw{xy}+x\blw{y}+\lambda xy   \mid x,y\in \mapm{X}\Big\}$$
is a Gr\"{o}bner-Shirshov basis in $\bfk\mapm{X}$. \mlabel{it:mdiff01r}

\item The set $${\rm Irr}(S_{\rm mdiffl}):= \Big\{\frakM(X)\setminus \{q|_{\blw{x}y}\mid x,y\in \frakM(X),q \in \frakM^\star(X)\}\Big\}$$
is a \bfk-basis of the free modified differential algebra $\bfk\frakM(X)/\Id(S_{\rm mdiffl})$ of weight $\lambda$. \mlabel{it:mdiff02r}
\end{enumerate}
\mlabel{thm:gsbmd0r}
\end{theorem}

\begin{proof}
\mref{it:mdiff01r} We can replace the $S_{\rm mdiffl}$ by
$$ S_{\rm mdiffl} = \Big\{\blw{x}y-\blw{xy}+x\blw{y}+\lambda xy, \blw{\bfone} +\lambda  \mid x,y\in \mapm{X}\setminus\{\bfone\}\Big\}.$$
Denote by
$$f(x,y) = \blw{x}y-\blw{xy}+x\blw{y} + \lambda xy\in S_{\rm mdiffl},\,\text{ where }\, x,y\in  \frakM(X)\setminus\{\bfone\}.$$
We have the leading monomial $\lbar{f(x,y)} = \blw{x}y.$
The ambiguities of all possible intersection compositions are
$$w_1=\blw{x}\blw{y}z, \text{ where } x,y,z \in \mapm{X}\setminus\{\bfone\}, $$
and the ambiguities of all possible including compositions are
$$w_2= \blw{q\suba{\blw{x}y}}z, w_3=\blw{x}q\suba{\blw{y}z}, w_4=\blw{q\suba{\blw{\bfone}}}x \,\text{ and }\,w_5=\blw{x} q|_{\blw{\bfone}},$$
where $x,y,z \in \mapm{X}\setminus\{\bfone\}$ and $q \in \frakM^\star(X)$.
For the ambiguity $w_1$, the corresponding intersection composition is trivial by
\begin{align*}
&\Big(f(x,\blw{y}), f(y,z)\Big)_{w_1}\\
=&~ \big(\blw{x}\blw{y}-\blw{x\blw{y}}+x\blw{\blw{y}}+\lambda x\blw{y}\big)z-\blw{x}\big(\blw{y}z-\blw{yz}+y\blw{z}+\lambda yz\big)\\
=&~-\blw{x\blw{y}}z+x\blw{\blw{y}}z+\lambda x\blw{y}z+\blw{x}\blw{yz}-\blw{x}y\blw{z}-\lambda\blw{x}yz\\
\equiv&~
-\blw{x\blw{y}z}+x\blw{y}\blw{z}+\lambda x\blw{y}z
+x\blw{\blw{y}z}-x\blw{y}\blw{z}-\lambda x\blw{y}z+\lambda x\blw{y}z\\
&+\blw{x\blw{yz}}-x\blw{\blw{yz}}-\lambda x\blw{yz}
-\blw{xy\blw{z}}+x\blw{y\blw{z}}+\lambda xy\blw{z}-\lambda \blw{x}yz\\
\equiv&~
-\blw{x\blw{y}z}+\lambda x\blw{y}z
+x\blw{\blw{y}z}+\blw{x\blw{yz}}-x\blw{\blw{yz}}-\lambda x\blw{yz}
-\blw{xy\blw{z}}+x\blw{y\blw{z}}+\lambda xy\blw{z}-\lambda \blw{x}yz\\
\equiv&~ \blw{-x\blw{yz}+xy\blw{z}+\lambda xyz}
+x\blw{\blw{yz}-y\blw{z}-\lambda yz}
+\blw{x\blw{yz}}-x\blw{\blw{yz}}-\blw{xy\blw{z}}+x\blw{y\blw{z}}\\
&+\lambda\big(x\blw{y}z-x\blw{yz}+xy\blw{z}-\blw{x}yz\big)\\
\equiv&~\lambda\big(\blw{xyz}-x\blw{yz}+x\blw{y}z-x\blw{yz}+xy\blw{z}-\blw{x}yz\big)\\
\equiv&~\lambda\big(\blw{xyz}+x\blw{y}z-2x\blw{yz}+xy\blw{z}-\blw{x}yz\big)\\
\equiv&~\lambda\big(\blw{xyz}+x\blw{yz}-xy\blw{z}-\lambda xyz-2x\blw{yz}+xy\blw{z}-\blw{xy}z+x\blw{y}z+\lambda xyz\big)\\
\equiv&~\lambda\big(\blw{xyz}-x\blw{yz}-\blw{xy}z+x\blw{y}z\big) \\
\equiv&~\lambda\big(\blw{xyz}-x\blw{yz}-\blw{xyz}+xy\blw{z}+\lambda xyz+x\blw{yz}-xy\blw{z}-\lambda xyz\big) \\
\equiv&~ 0 \mod (S_{\rm mdiffl}, w_1).
\end{align*}

The ambiguity $w_2$ is from the including composition
$$\Big(f(q\suba{\blw{x}y},z), f(x,y)\Big)_{w_2}.$$
Its triviality follows from
\begin{align*}
&\Big(f(q\suba{\blw{x}y},z), f(x,y)\Big)_{w_2}\\
=&~\Big(\blw{q\suba{\blw{x}y}}z-\blw{q\suba{\blw{x}y}z}+q\suba{\blw{x}y}\blw{z}+\lambda q\suba{\blw{x}y}{z}\Big)
-\Big(\blw{q\suba{\blw{x}y-\blw{xy}+x\blw{y}+\lambda xy}}z\Big)\\
=&~-\blw{q\suba{\blw{x}y}z}+q\suba{\blw{x}y}\blw{z}+\lambda q\suba{\blw{x}y}{z}
+\blw{q\suba{\blw{xy}}}z-\blw{q\suba{x\blw{y}}}z-\lambda\blw{q\suba{x{y}}}z\\
\equiv&~-\blw{q\suba{\blw{xy}-x\blw{y}-\lambda xy}z}+q\suba{\blw{xy}-x\blw{y}-\lambda xy}\blw{z}+\lambda q\suba{\blw{xy}-x\blw{y}-\lambda xy}{z}
+\blw{q\suba{\blw{xy}}z}-q\suba{\blw{xy}}\blw{z}-\lambda q\suba{\blw{xy}}{z}\\
&-\blw{q\suba{x\blw{y}}z}+q\suba{x\blw{y}}\blw{z}+\lambda q\suba{x\blw{y}}{z}-\lambda\blw{q\suba{x{y}}}z\\
\equiv&~\lambda\big(\blw{q\suba{ xy}z}- q\suba{xy}\blw{z}+ q\suba{\blw{xy}-x\blw{y}-\lambda xy}{z}
- q\suba{\blw{xy}}{z}+ q\suba{x\blw{y}}{z}-\blw{q\suba{x{y}}}z \big)\\
\equiv&~\lambda\big(\blw{q\suba{ xy}z}- q\suba{xy}\blw{z}-\lambda q\suba{xy}{z}-\blw{q\suba{x{y}}z}+{q\suba{x{y}}}\blw{z}+\lambda{q\suba{x{y}}}z \big)\\
\equiv&~0 \mod (S_{\rm mdiffl}, w_2)
\end{align*}

For the ambiguity $w_3$, its composition is trivial by
\begin{align*}
&\Big(f(x, q\suba{\blw{y}z}), f(y,z)\Big)_{w_3}\\
=&~\Big(\blw{x}q\suba{\blw{y}z}-\blw{xq\suba{\blw{y}z}}+x\blw{q\suba{\blw{y}z}} + \lambda xq\suba{\blw{y}z}\Big)
-\Big(\blw{x}q\suba{\blw{y}z-\blw{yz}+y\blw{z} + \lambda yz}\Big)\\
=&~-\blw{xq\suba{\blw{y}z}}+x\blw{q\suba{\blw{y}z}} + \lambda xq\suba{\blw{y}z}
+\blw{x}q\suba{\blw{yz}}-\blw{x}q\suba{y\blw{z}}-\blw{x}q\suba{\lambda yz}\\
\equiv&~-\blw{xq\suba{\blw{yz}-y\blw{z} - \lambda yz}}+x\blw{q\suba{\blw{yz}-y\blw{z} - \lambda yz}} + \lambda xq\suba{\blw{yz}-y\blw{z} - \lambda yz}
+\blw{xq\suba{\blw{yz}}}-x\blw{q\suba{\blw{yz}}} - \lambda xq\suba{\blw{yz}} \\
&-\blw{xq\suba{y\blw{z}}}+x\blw{q\suba{y\blw{z}}} + \lambda xq\suba{y\blw{z}}
-\blw{x}q\suba{\lambda yz}\\
=&~\lambda\big(\blw{xq\suba{ yz}}-x\blw{q\suba{ yz}} + xq\suba{\blw{yz}-y\blw{z} - \lambda yz}
- xq\suba{\blw{yz}}  +  xq\suba{y\blw{z}}-\blw{x}q\suba{ yz}\big)\\
\equiv&~\lambda\big(\blw{xq\suba{ yz}}-x\blw{q\suba{ yz}} - \lambda xq\suba{ yz}-\blw{xq\suba{ yz}}+x\blw{q\suba{ yz}} + \lambda xq\suba{ yz}\big)\\
\equiv&~0 \mod (S_{\rm mdiffl}, w_3)
\end{align*}

The composition of the ambiguity $w_4$ is trivial by
\begin{align*}
&\Big(f(q\suba{\blw{\bfone}},x), \blw{\bfone}+\lambda\Big)\\
=&~ (\blw{q\suba{\blw{\bfone}}}x-\blw{q\suba{\blw{\bfone}}x}+q\suba{\blw{\bfone}}\blw{x} + \lambda q\suba{\blw{\bfone}}x)-\blw{q\suba{\blw{\bfone}+\lambda}}x\\
=&~ -\blw{q\suba{\blw{\bfone}}x}+q\suba{\blw{\bfone}}\blw{x} + \lambda q\suba{\blw{\bfone}}x-\blw{q\suba{\lambda}}x\\
\equiv&~ \blw{q\suba{\lambda}x}-q\suba{\lambda}\blw{x} - \lambda q\suba{\lambda}x
-\blw{q\suba{\lambda}x}+q\suba{\lambda}\blw{x} + \lambda q\suba{\lambda}x\\
\equiv&~ 0 \mod (S_{\rm mdiffl}, w_4).
\end{align*}

Finally, the triviality of the composition from the ambiguity $w_5$
is from
\begin{align*}
&\Big(f(x, q\suba{\blw{\bfone}}), \blw{\bfone}+\lambda\Big)\\
=&~ (\blw{x}q\suba{\blw{\bfone}}-\blw{xq\suba{\blw{\bfone}}}+x\blw{q\suba{\blw{\bfone}}} + \lambda xq\suba{\blw{\bfone}} )-\blw{x}q\suba{\blw{\bfone}+\lambda}\\
=&~ -\blw{xq\suba{\blw{\bfone}}}+x\blw{q\suba{\blw{\bfone}}} + \lambda xq\suba{\blw{\bfone}} -\blw{x}q\suba{\lambda}\\
\equiv&~ \blw{xq\suba{\lambda}}-x\blw{q\suba{\lambda}} - \lambda xq\suba{\lambda} -\blw{xq\suba{\lambda}}+x\blw{q\suba{\lambda}} + \lambda xq\suba{\lambda}\\
\equiv&~ 0 \mod (S_{\rm mdiffl}, w_5).
\end{align*}
Thus $S_{\rm mdiffl}$ is a Gr\"{o}bner-Shirshov basis in $\bfk\mapm{X}$.

\mref{it:mdiff02r} It follows from Lemma~\mref{lem:cdl} and Item~\mref{it:mdiff01r}.
\end{proof}

Replacing the monomial $\dbl$ utilized in Theorem~\mref{thm:gsbmd0r} by $\dbr$, we have

\begin{theorem} Let $X$ be a well-ordered set and $\dbr$ the monomial order on $\frakM(X)$ defined in Eq.~\meqref{eq:ordr}.
\begin{enumerate}
\item The set
$$S_{\rm mdiffr} := \Big\{x\blw{y}-\blw{xy}+\blw{x}y+\lambda xy  \mid x,y\in \mapm{X}\Big\}$$
is a Gr\"{o}bner-Shirshov basis in $\bfk\mapm{X}$. \mlabel{it:mdiff01l}

\item The set $${\rm Irr}(S_{\rm mdiffr}):= \Big\{\frakM(X)\setminus \{q|_{x\blw{y}} \mid x,y\in \frakM(X),q \in \frakM^\star(X)\}\Big\}$$
is a \bfk-basis of the free modified differential algebra $\bfk\frakM(X)/\Id(S_{\rm mdiffr})$ of weight $\lambda$. \mlabel{it:mdiff02l}
\end{enumerate}
\mlabel{thm:gsbmd0l}
\end{theorem}

\begin{proof}
It is similar to the proof of  Theorem~\mref{thm:gsbmd0r}.
\end{proof}

\section{New bases of free (modified) Rota-Baxter algebras }
In this section, we construct respectively some new linear bases of free nonunitary Rota-Baxter algebras and free
nonunitary modified Rota-Baxter algebras, using the method of Gr\"{o}bner-Shirshov bases.

\subsection{New bases of free nonunitary Rota-Baxter algebras}
We are going to give the first new linear basis of the free nonunitary Rota-Baxter algebra.
\begin{defn}\cite{Ba,Gl}
Let $\lambda\in \bfk$ be a fixed element. A {\bf Rota-Baxter algebra of weight $\lambda$ } is a pair $(A, P)$ consisting of an algebra $A$ and a linear operator $P: A \longrightarrow A$ such that
$$P(x)P(y)=P(P(x)y)+P(xP(y))+\lambda P(xy), \,\text{ for }\, x,y \in A.$$
\end{defn}

Now we make clear the reason of restriction to the nonunitary case.
If we employ the monomial order $\leq_{\rm db}$ in~\mcite{ZGGS}, then the leading monomial of the Rota-Baxter identity
is $\blw{x}\blw{y}$ and the set $$\Big\{ \blw{x}\blw{y} - \blw{\blw{x} y} - \blw{x\blw{y}} - \lambda \blw{xy}\mid x,y\in \frakM(X)\Big\}$$
is a Gr\"{o}bner-Shirshov basis in $\bfk\frakM(X)$~\mcite{ZGGS}.
If we utilize the monomial order $\dbl$ (resp. $\dbr$), then the leading monomial of the Rota-Baxter identity is $\blw{\blw{x} y}$ (resp. $\blw{x\blw{y}}$) for $x,y \in \mapm{X}\setminus\{\bfone\}$. Indeed, we have
\begin{align*}
\degx(\blw{x\blw{y}})&= \degx(\blw{\blw{x} y})=\degx(\blw{x}\blw{ y})=\degx(\blw{x y})=\degx(x)+\degx(y),\\
\patl(\blw{\blw{x} y})&=\Big(P\mu P \,\patl(x), P\mu \,\patl(y)\Big)\,\text{ and }\,\patr(\blw{\blw{x} y})=\Big( P\mu \,\patr(y), P\mu P \,\patr(x\Big),\\
 \patl(\blw{x\blw{y}})&=\Big(P\mu\, \patl(x), P\mu P \,\patl(y)\Big)\,\text{ and }\,\patr(\blw{x\blw{y}})=\Big(P\mu P \,\patr(y),P\mu \,\patr(x)\Big),\\
 \patl(\blw{x}\blw{ y})&=\Big(\mu P\,\patl(x),  \mu P\,\patl(y)\Big)\,\text{ and }\, \patr(\blw{x}\blw{ y})=\Big( \mu P\,\patr(y), \mu P\,\patr(x)\Big),\\
\patl(\blw{x y})&=\Big(P\mu\, \patl(x), P\mu \,\patl(y)\Big)\,\text{ and }\, \patr(\blw{xy})=\Big( P\mu \,\patr(y), P\mu \,\patr(x)\Big).
\end{align*}
By Eqs.~(\mref{eq:ordl}) and~(\mref{eq:ordr}),
$$\blw{\blw{x} y}\dbln\blw{x}\blw{ y} \dbln\blw{x\blw{y}} \dbln \blw{x y}$$
and
$$\blw{x\blw{y}}\dbrn\blw{x}\blw{ y} \dbrn\blw{\blw{x} y} \dbrn \blw{x y}.$$
However the set
$$\Big\{\blw{\blw{x}y} + \blw{x\blw{y}} +\lambda \blw{xy}- \blw{x}\blw{y}\mid x,y\in \mapm{X}\Big\}$$
is not a Gr\"{o}bner-Shirshov basis in $\bfk\frakM(X)$ with respect to $\dbl$ or $\dbr$, explained as follows.
Without loss of generality, we consider the monomial order $\dbl$.
The leading monomial of $$\blw{x}\blw{\bfone}-\blw{x\blw{\bfone}}-\blw{\blw{x}}-\lambda \blw{x},\,\text{ where } x\in \frakM(X)\setminus\{\bfone\},$$ is $\blw{x}\blw{\bfone}$, and the leading monomial of $\blw{\bfone}\blw{\bfone}-2\blw{\blw{\bfone}}-\lambda \blw{\bfone}$ is $\blw{\bfone}\blw{\bfone}$. For the ambiguity $\blw{x}\blw{\bfone}\blw{\bfone}$, the corresponding intersection composition is
\begin{align*}
&\Big( \blw{x}\blw{\bfone}-\blw{x\blw{\bfone}}-\blw{\blw{x}}-\lambda \blw{x}, \blw{\bfone}\blw{\bfone}-2\blw{\blw{\bfone}}-\lambda \blw{\bfone}\Big)_{\blw{x}\blw{\bfone}\blw{\bfone}}\\
=~& \Big(\blw{x}\blw{\bfone}-\blw{x\blw{\bfone}}-\blw{\blw{x}}-\lambda \blw{x}\Big)\blw{\bfone}-
\blw{x}\Big(\blw{\bfone}\blw{\bfone}-2\blw{\blw{\bfone}}-\lambda \blw{\bfone}\Big)\\
=~&-\Big(\blw{x\blw{\bfone}}\blw{\bfone}+\blw{\blw{x}}\blw{\bfone}\Big)
+2\blw{x}\blw{\blw{\bfone}}\\
\equiv~&-\Big(
\blw{x\blw{\bfone}\blw{\bfone}}+\blw{\blw{x\blw{\bfone}}}+\lambda \blw{x\blw{\bfone}}
+\blw{\blw{x}\blw{\bfone}}+\blw{\blw{\blw{x}}}+\lambda \blw{\blw{x}}\Big)
+2\blw{x}\blw{\blw{\bfone}}\\
\equiv~&-\Big(
\Big\lc 2x\blw{\blw{\bfone}}+\lambda x\blw{\bfone}\Big\rc+\blw{\blw{x\blw{\bfone}}}+\lambda \blw{x\blw{\bfone}}
+ \Big\lc\blw{x\blw{\bfone}}+\blw{\blw{x}}\\
&+\lambda \blw{x}\Big\rc +\blw{\blw{\blw{x}}}+\lambda \blw{\blw{x}}\Big)
+2\blw{x}\blw{\blw{\bfone}}\\
\equiv~&-\Big(
2\blw{x\blw{\blw{\bfone}}}+\lambda \blw{x\blw{\bfone}}+\blw{\blw{x\blw{\bfone}}}+\lambda \blw{x\blw{\bfone}}
+\blw{\blw{x\blw{\bfone}}}+\blw{\blw{\blw{x}}}\\
&+\lambda \blw{\blw{x}}+\blw{\blw{\blw{x}}}+\lambda \blw{\blw{x}}\Big)
+2\blw{x}\blw{\blw{\bfone}}\\
\equiv~&-2\Big(
\blw{x\blw{\blw{\bfone}}}+\lambda \blw{x\blw{\bfone}}+\blw{\blw{x\blw{\bfone}}}+\blw{\blw{\blw{x}}}+\lambda \blw{\blw{x}}\Big)
+2\blw{x}\blw{\blw{\bfone}},
\end{align*}
which is not trivial.

Roughly speaking, $\maps{X}$ is the subset of $\mapm{X}$ not involving $\bfone$.
For example, $\bfone, \lc \bfone\rc, x\lc \bfone\rc\notin \maps{X}$ with $x\in X$.
Notice that the restriction of the monomial orders $\dbl$ and $\dbr$ to $\maps{X}$ are still monomial orders.

\begin{theorem} Let $(X, \leq)$ be a well-ordered set and $\dbl$ the monomial order on $\maps{X}$.
\begin{enumerate}
\item The set
$$ S_{{\rm rbl}}:= \Big\{\blw{\blw{x}y} + \blw{x\blw{y}} +\lambda \blw{xy}- \blw{x}\blw{y}\mid x,y\in \maps{X}\Big\} $$
is a Gr\"{o}bner-Shirshov basis in $\bfk\maps{X}$. \mlabel{it:rbr1}

\item The set $${\rm Irr}(S_{{\rm rbl}}) := \Big\{\frakS(X) \setminus \{q|_{\blw{\blw{x}y}} \mid x,y\in \frakS(X), q \in \frakS^\star(X)\} \Big\}$$ is a \bfk-basis of the
free nonunitary Rota-Baxter algebra $\bfk\frakS(X)/\Id(S_{{\rm rbl}})$ on $X$. \mlabel{it:rbr2}
\end{enumerate}
\mlabel{thm:rbgs}
\end{theorem}

\begin{proof}
\mref{it:rbr1} Denote
$$f(x,y) =\blw{\blw{x}y} + \blw{x\blw{y}} +\lambda \blw{xy}- \blw{x}\blw{y}\in S_{{\rm rbl}},\,\text{ where }\,  x,y\in \maps{X}.$$
By Eq.~(\mref{eq:ordl}), we have the leading monomial
$$\lbar{f(x,y)}=\blw{\blw{x}y}.$$
Sine $|\lbar{f(x,y)} | = 1$, there are no intersection compositions.
The ambiguities of all including possible compositions are
 $$w_1=\blw{\blw{q\suba{{\blw{\blw{x}y}}}}z}\,\text{ and }\, w_2=\blw{\blw{x}q\suba{{\blw{\blw{y}z}}}},$$
where $x,y,z \in \maps{X}$ and $q \in \frakS^\star(X)$.

For the ambiguity $w_1$, it is from the including composition
$$
\Big(f(q\suba{{\blw{\blw{x}y}}},z), f(x,y)\Big)_{w_1},
$$
and its triviality follows from
\begin{align*}
&\Big(f(q\suba{{\blw{\blw{x}y}}},z), f(x,y)\Big)_{w_1}\\
=~&\Big(\blw{\blw{q\suba{{\blw{\blw{x}y}}}}z}+ \blw{q\suba{{\blw{\blw{x}y}}}\blw{z}} +\lambda \blw{q\suba{{\blw{\blw{x}y}}}z}- \blw{q\suba{{\blw{\blw{x}y}}}}\blw{z}\Big)\\
&-\Big(\blw{\blw{q\suba{{\blw{\blw{x}y}}}}z}+\blw{\blw{q\suba{{\blw{x\blw{y}}}}}z}+\blw{\blw{q\suba{{\lambda \blw{xy}}}}z}-\blw{\blw{q\suba{{ \blw{x}\blw{y}}}}z}\Big)\\
=~& \blw{q\suba{{\blw{\blw{x}y}}}\blw{z}} +\lambda \blw{q\suba{{\blw{\blw{x}y}}}z}- \blw{q\suba{{\blw{\blw{x}y}}}}\blw{z}-\blw{\blw{q\suba{{\blw{x\blw{y}}}}}z}-\blw{\blw{q\suba{{\lambda \blw{xy}}}}z}+\blw{\blw{q\suba{{ \blw{x}\blw{y}}}}z}\\
\equiv~& \blw{q\suba{-{\blw{x\blw{y}} -\lambda \blw{xy}+ \blw{x}\blw{y}}}\blw{z}}+\lambda \blw{q\suba{{ -\blw{x\blw{y}} -\lambda \blw{xy}+ \blw{x}\blw{y}}}z}- \blw{q\suba{{- \blw{x\blw{y}} -\lambda \blw{xy}+ \blw{x}\blw{y}}}}\blw{z}\\
&+\blw{q\suba{{\blw{x\blw{y}}}}\blw{z}}+\lambda\blw{q\suba{{\blw{x\blw{y}}}}z}-\blw{q\suba{{\blw{x\blw{y}}}}}\blw{z}
+\blw{q\suba{{\lambda \blw{xy}}}\blw{z}}+\lambda\blw{q\suba{{\lambda \blw{xy}}}z}-\blw{q\suba{{\lambda \blw{xy}}}}\blw{z}\\
&-\blw{q\suba{{ \blw{x}\blw{y}}}\blw{z}}-\lambda\blw{q\suba{{ \blw{x}\blw{y}}}z}+\blw{q\suba{{ \blw{x}\blw{y}}}}\blw{z}\\
\equiv~& 0 \mod(S_{\rm rbl},w_1).
\end{align*}

Further the ambiguity $w_2$ is from the including composition $$\Big(f(x,q\suba{{\blw{\blw{y}z}}}), f(y,z)\Big)_{w_2}, $$
whose triviality follows from
\begin{align*}
&\Big(f(x,q\suba{{\blw{\blw{y}z}}}), f(y,z)\Big)_{w_2}\\
=~&\Big(\blw{\blw{x}q\suba{{\blw{\blw{y}z}}}} + \blw{x\blw{q\suba{{\blw{\blw{y}z}}}}} +\lambda \blw{xq\suba{{\blw{\blw{y}z}}}}- \blw{x}\blw{q\suba{{\blw{\blw{y}z}}}}\Big)\\
&-\Big(\blw{\blw{x}q\suba{{\blw{\blw{y}z} + \blw{y\blw{z}} +\lambda \blw{yz}- \blw{y}\blw{z}}}} \Big)\\
=~& \blw{x\blw{q\suba{{\blw{\blw{y}z}}}}} +\lambda \blw{xq\suba{{\blw{\blw{y}z}}}}- \blw{x}\blw{q\suba{{\blw{\blw{y}z}}}}
-\blw{\blw{x}q\suba{{  \blw{y\blw{z}}}}} -\blw{\blw{x}q\suba{\lambda \blw{yz}}} + \blw{\blw{x}q\suba{\blw{y}\blw{z}}} \\
\equiv~& \blw{x\blw{q\suba{{- \blw{y\blw{z}} -\lambda \blw{yz}+ \blw{y}\blw{z}}}}} +\lambda \blw{xq\suba{{- \blw{y\blw{z}} -\lambda \blw{yz}+ \blw{y}\blw{z}}}}- \blw{x}\blw{q\suba{{- \blw{y\blw{z}} -\lambda \blw{yz}+ \blw{y}\blw{z}}}}\\
&+ \blw{x\blw{q\suba{{  \blw{y\blw{z}}}}}} +\lambda \blw{xq\suba{{  \blw{y\blw{z}}}}}- \blw{x}\blw{q\suba{{  \blw{y\blw{z}}}}}
+ \blw{x\blw{q\suba{\lambda \blw{yz}}}} + \lambda \blw{xq\suba{\lambda \blw{yz}}} - \blw{x}\blw{q\suba{\lambda \blw{yz}}} \\
& - \blw{x\blw{q\suba{\blw{y}\blw{z}}}} - \lambda \blw{xq\suba{\blw{y}\blw{z}}}+ \blw{x}\blw{q\suba{\blw{y}\blw{z}}}\\
\equiv~& 0 \mod(S_{\rm rbl},w_2).
\end{align*}
Thus $S_{{\rm rbl}}$ is a Gr\"{o}bner-Shirshov basis in $\bfk\maps{X}$.

\mref{it:rbr2} It follows from Lemma~\mref{lem:cdlnu} and Item~\mref{it:rbr1}.
\end{proof}

\begin{remark}
Deotsenko~\mcite{Deo} gave an operad version of Theorem~\mref{thm:rbgs}~\mref{it:rbr2}, employing
the path-lexicographical order~\mcite{He}.
\end{remark}

Now we turn into the second new linear basis of the free nonunitary Rota-Baxter algebra.

\begin{theorem} Let $(X, \leq)$ be a well-ordered set and $\dbr$ the monomial order on $\maps{X}$.
\begin{enumerate}
\item The set
$$ S_{{\rm rbr}} := \Big\{\blw{x\blw{y}} +\blw{\blw{x}y} + \lambda \blw{xy}- \blw{x}\blw{y}\mid x,y\in \maps{X}\Big\} $$
is a Gr\"{o}bner-Shirshov basis in $\bfk\maps{X}$. \mlabel{it:rbl1}

\item The set $${\rm Irr}(S_{{\rm rbr}}) := \Big\{\frakS(X) \setminus \{q|_{\blw{x\blw{y}}}  \mid  x,y\in \frakS(X), q \in \frakS^\star(X)\} \Big\}$$ is a \bfk-basis of the
free nonunitary Rota-Baxter algebra $\bfk\frakS(X)/\Id(S_{{\rm rbr}})$ on $X$. \mlabel{it:rbl2}
\end{enumerate}
\mlabel{thm:rbgs2}
\end{theorem}

\begin{proof}
It is analogous to the proof of Theorem~\mref{thm:rbgs}.
\end{proof}

\subsection{New bases of free nonunitary modified Rota-Baxter algebras}
Let us agree to begin with the concept of modified Rota-Baxter algebras.
\begin{defn}\mcite{Tri,BGN}
Let $\lambda\in\bfk$ be a fixed constant. A {\bf modified Rota-Baxter algebra of weight $\lambda$} is an algebra $A$ together with a linear operator $P: A \longrightarrow A$ such that
$$P(x)P(y)=P(P(x)y+xP(y))+\lambda xy,\,\text{ for } x,y\in A.$$
\end{defn}

Using the two monomial orders defined in Eqs.~\meqref{eq:ordl} and~\meqref{eq:ordr}, we can obtain two new linear bases of the free nonunitary modified Rota-Baxter algebra.

\begin{theorem} Let $(X, \leq)$ be a well-ordered set and $\dbl$ the monomial order on $\maps{X}$.
\begin{enumerate}
\item The set
$$ S_{{\rm mrbl}}:= \Big\{\blw{\blw{x}y} + \blw{x\blw{y}} +\lambda {xy}- \blw{x}\blw{y}\mid x,y\in \maps{X}\Big\} $$
is a Gr\"{o}bner-Shirshov basis in $\bfk\maps{X}$. \mlabel{it:mrbr1}

\item The set $${\rm Irr}(S_{{\rm mrbl}}) := \Big\{\frakS(X) \setminus \{q|_{\blw{\blw{x}y}} \mid x,y\in \frakS(X), q \in \frakS^\star(X)\} \Big\}$$ is a \bfk-basis of the
free nonunitary modified Rota-Baxter algebra $\bfk\frakS(X)/\Id(S_{{\rm mrbl}})$ on $X$. \mlabel{it:mrbr2}
\end{enumerate}
\mlabel{thm:mrbgs}
\end{theorem}

\begin{proof}
\mref{it:rbr1} We set
$$f(x,y) =\blw{\blw{x}y} + \blw{x\blw{y}} +\lambda xy- \blw{x}\blw{y}\in S_{\rm mrbl}.$$
In analogy to the proof of Theorem~\mref{thm:rbgs}, we have
$$\lbar{f(x,y)}=\blw{\blw{x}y}.$$
There are no ambiguities of intersection compositions by $|\lbar{f(x,y)}| = 1$.
The ambiguities of all possible compositions are
$$
w_1=\blw{\blw{q\suba{{\blw{\blw{x}y}}}}z}\,\text{ and }\,w_2=\blw{\blw{x}q\suba{{\blw{\blw{y}z}}}}
$$
where $x,y,z \in \maps{X}$ and $q \in \frakS^\star(X)$. The ambiguity $w_1$ is from the including composition
$$
\Big(f(q\suba{{\blw{\blw{x}y}}},z), f(x,y)\Big)_{w_1},
$$
whose triviality follows from
\begin{align*}
&\Big(f(q\suba{{\blw{\blw{x}y}}},z), f(x,y)\Big)_{w_1}\\
=~&\Big(\blw{\blw{q\suba{{\blw{\blw{x}y}}}}z}+ \blw{q\suba{{\blw{\blw{x}y}}}\blw{z}} +\lambda{q\suba{{\blw{\blw{x}y}}}z}- \blw{q\suba{{\blw{\blw{x}y}}}}\blw{z}\Big)\\
&-\Big(\blw{\blw{q\suba{{\blw{\blw{x}y}}}}z}+\blw{\blw{q\suba{{\blw{x\blw{y}}}}}z}+\blw{\blw{q\suba{{\lambda{xy}}}}z}-\blw{\blw{q\suba{{ \blw{x}\blw{y}}}}z}\Big)\\
=~& \blw{q\suba{{\blw{\blw{x}y}}}\blw{z}} +\lambda{q\suba{{\blw{\blw{x}y}}}z}- \blw{q\suba{{\blw{\blw{x}y}}}}\blw{z}-\blw{\blw{q\suba{{\blw{x\blw{y}}}}}z}-\blw{\blw{q\suba{{\lambda{xy}}}}z}+\blw{\blw{q\suba{{ \blw{x}\blw{y}}}}z}\\
\equiv~& \blw{q\suba{-{\blw{x\blw{y}} -\lambda{xy}+ \blw{x}\blw{y}}}\blw{z}}+\lambda{q\suba{{- \blw{x\blw{y}} -\lambda{xy}+ \blw{x}\blw{y}}}z}+ \blw{q\suba{{ \blw{x\blw{y}} +\lambda{xy}- \blw{x}\blw{y}}}}\blw{z}\\
&+\blw{q\suba{{\blw{x\blw{y}}}}\blw{z}}+\lambda\blw{q\suba{{\blw{x\blw{y}}}}z}-\blw{q\suba{{\blw{x\blw{y}}}}}\blw{z}
+\blw{q\suba{{\lambda{xy}}}\blw{z}}+\lambda\blw{q\suba{{\lambda{xy}}}z}-\blw{q\suba{{\lambda{xy}}}}\blw{z}\\
&-\blw{q\suba{{ \blw{x}\blw{y}}}\blw{z}}-\lambda\blw{q\suba{{ \blw{x}\blw{y}}}z}+\blw{q\suba{{ \blw{x}\blw{y}}}}\blw{z}\\
\equiv~& 0 \mod(S_{\rm mrbl},w_1).
\end{align*}

The ambiguity $w_2$ is from the including composition $$\Big(f(x,q\suba{{\blw{\blw{y}z}}}), f(y,z)\Big)_{w_2},$$
which is trivial by
\begin{align*}
&\Big(f(x,q\suba{{\blw{\blw{y}z}}}), f(y,z)\Big)_{w_2}\\
=~&\Big(\blw{\blw{x}q\suba{{\blw{\blw{y}z}}}} + \blw{x\blw{q\suba{{\blw{\blw{y}z}}}}} +\lambda {xq\suba{{\blw{\blw{y}z}}}}- \blw{x}\blw{q\suba{{\blw{\blw{y}z}}}}\Big)\\
&-\Big(\blw{\blw{x}q\suba{{\blw{\blw{y}z} + \blw{y\blw{z}} +\lambda {yz}- \blw{y}\blw{z}}}} \Big)\\
=~& \blw{x\blw{q\suba{{\blw{\blw{y}z}}}}} +\lambda {xq\suba{{\blw{\blw{y}z}}}}- \blw{x}\blw{q\suba{{\blw{\blw{y}z}}}}
-\blw{\blw{x}q\suba{{  \blw{y\blw{z}}}}} - \blw{\blw{x}q\suba{\lambda {yz}}} + \blw{\blw{x}q\suba{\blw{y}\blw{z}}} \\
\equiv~& \blw{x\blw{q\suba{{- \blw{y\blw{z}} -\lambda {yz}+ \blw{y}\blw{z}}}}} +\lambda {xq\suba{{- \blw{y\blw{z}} -\lambda {yz}+ \blw{y}\blw{z}}}}- \blw{x}\blw{q\suba{{- \blw{y\blw{z}} -\lambda {yz}+ \blw{y}\blw{z}}}}\\
&+ \blw{x\blw{q\suba{{  \blw{y\blw{z}}}}}} +\lambda {xq\suba{{  \blw{y\blw{z}}}}}- \blw{x}\blw{q\suba{{  \blw{y\blw{z}}}}}
+ \blw{x\blw{q\suba{\lambda {yz}}}} + \lambda {xq\suba{\lambda {yz}}} - \blw{x}\blw{q\suba{\lambda {yz}}} \\
& - \blw{x\blw{q\suba{\blw{y}\blw{z}}}} - \lambda {xq\suba{\blw{y}\blw{z}}} + \blw{x}\blw{q\suba{\blw{y}\blw{z}}}\\
\equiv~& 0 \mod(S_{\rm rbl},w_2).
\end{align*}
Thus $S_{{\rm mrbl}}$ is a Gr\"{o}bner-Shirshov basis in $\bfk\maps{X}$.

\mref{it:rbr2} It follows from Lemma~\mref{lem:cdlnu} and Item~\mref{it:mrbr1}.
\end{proof}

\begin{theorem} Let $(X, \leq)$ be a well-ordered set and $\dbr$ the monomial order on $\maps{X}$.
\begin{enumerate}
\item The set
$$ S_{{\rm mrbr}} := \Big\{\blw{x\blw{y}} +\blw{\blw{x}y} + \lambda {xy}- \blw{x}\blw{y}\mid x,y\in \maps{X}\Big\} $$
is a Gr\"{o}bner-Shirshov basis in $\bfk\maps{X}$. \mlabel{it:mrbl1}

\item The set $${\rm Irr}(S_{{\rm mrbr}}) := \Big\{\frakS(X) \setminus \{q|_{\blw{x\blw{y}}} \mid x,y\in \frakS(X), q \in \frakS^\star(X)\} \Big\}$$ is a \bfk-basis of the
free nonunitary modified Rota-Baxter algebra $\bfk\frakS(X)/\Id(S_{{\rm mrbr}})$ on $X$. \mlabel{it:mrbl2}
\end{enumerate}
\mlabel{thm:mrbgs2}
\end{theorem}

\begin{proof}
In analogy to the proof of Theorem~\mref{thm:mrbgs}, the result holds.
\end{proof}

\smallskip

\noindent
{\bf Acknowledgements}:
This work was supported by the National Natural Science Foundation
of China (Grant No.12071191), the Natural Science Foundation of Gansu Province (Grant No. 20JR5RA249).

\end{document}